\documentclass{article}

\usepackage[english]{babel}

\usepackage{amsmath,amsthm,amssymb,dsfont,subcaption,tikz,hyperref,enumitem}

\usepackage{csquotes} 

\usepackage{scalerel}
\usepackage[usestackEOL]{stackengine}

\usepackage[tmargin=2.5cm,lmargin=2cm,rmargin=3cm,bmargin=2.5cm]{geometry}

\DeclareMathOperator\B{B}
\newcommand{\C}{\mathrm{C}}
\newcommand*{\coleq}{\mathrel{\vcenter{\baselineskip0.5ex\lineskiplimit0pt\hbox{\normalsize.}\hbox{\normalsize.}}}=}
\newcommand*{\eqcol}{=\mathrel{\vcenter{\baselineskip0.5ex\lineskiplimit0pt\hbox{\normalsize.}\hbox{\normalsize.}}}}
\renewcommand{\d}{\,\mathrm{d}}
\newcommand{\dx}{\,\mathrm{d}x}
\newcommand{\dxx}{\mathrm{d}x}
\newcommand{\dxxx}{\hspace{-0.2cm}\mathrm{d}x}
\renewcommand{\L}{\mathrm{L}}
\newcommand{\loc}{\mathrm{loc}}
\renewcommand{\P}{\mathrm{P}}
\newcommand{\R}{\mathds{R}}
\newcommand\sd{\,\triangle\,}
\newcommand{\W}{\mathrm{W}}
\newcommand{\F}{\mathcal{F}}
\newcommand{\D}{\mathrm{D}}
\newcommand{\cpt}{\mathrm{cpt}}
\def\dashint{\,\ThisStyle{\ensurestackMath{%
            \stackinset{c}{.2\LMpt}{c}{.5\LMpt}{\SavedStyle-}{\SavedStyle\phantom{\int}}}%
        \setbox0=\hbox{$\SavedStyle\int\,$}\kern-\wd0}\int}
\renewcommand{\c}{\mathrm{c}}
\newcommand{\A}{\mathcal{A}}
\newcommand{\G}{\mathcal{G}}

\theoremstyle{plain}
\newtheorem{theo}{Theorem}[section]
\newtheorem{defi}[theo]{Definition}
\newtheorem{prop}[theo]{Proposition}
\newtheorem{lemma}[theo]{Lemma}
\newtheorem{cor}[theo]{Corollary}
\newtheorem{definition}[theo]{Definition}

\theoremstyle{definition}
\newtheorem{setting}{Setting}

\theoremstyle{remark}
\newtheorem{rem}[theo]{Remark}

\newcommand{\Q}{\mathrm{Q}}
\newcommand{\lc}{\lambda}
\newcommand{\M}{\mathcal{M}}
\newcommand{\gc}{\Lambda}
\newcommand{\dgc}{{\Lambda_\D}}
\newcommand{\zob}{\Gamma}
\newcommand{\ahe}{\alpha}
\newcommand{\fhe}{\delta}
\renewcommand{\b}{\beta}
\newcommand{\q}{q}
\renewcommand{\sb}{{s_\beta}}
\newcommand{\sq}{{s_q}}
\newcommand{\sdb}{\tilde{\Lambda}}
\newcommand{\regreg}{\Omega_\mathrm{reg}}

\newcommand{\minhe}{{\min\{\ahe,\fhe\}}}

\numberwithin{equation}{section}

\title{Partial regularity for variational integrals with Morrey-Hölder zero-order terms, and the limit exponent in Massari's regularity theorem}
\author{Thomas Schmidt\footnote{Fachbereich Mathematik, Universität Hamburg, Bundesstr.\@ 55, 20146 Hamburg, Germany. Email addresses: \href{mailto:thomas.schmidt.math@uni-hamburg.de}{thomas.schmidt.math@uni-hamburg.de}, \href{mailto:jule.schuett@uni-hamburg.de}{jule.schuett@uni-hamburg.de}.}\qquad\qquad Jule Helena Schütt${}^\ast$}
\date{\today}

\begin{document}

\maketitle

\begin{abstract}
We revisit the partial $\C^{1,\alpha}$ regularity theory for minimizers of non-parametric integrals with emphasis on sharp dependence of the Hölder exponent $\alpha$ on structural assumptions for general zero-order terms. A particular case of our conclusions carries over to the parametric setting of Massari's regularity theorem for prescribed-mean-curvature hypersurfaces and there confirms optimal regularity up to the limit exponent.
\end{abstract}


\setcounter{tocdepth}{1}
\tableofcontents

\section{Introduction}

In this paper, for $n,N\in \mathds{N}$ and bounded open $\Omega\subseteq \R^n$, we further develop the partial $\C^{1,\ahe}$ regularity theory for minimizers of variational integrals
\begin{equation}\label{DEF: integral functional}
  \F[w]\coleq  \int_\Omega\big[f(\D w)+g(\,\cdot\,,w)\big]\dx
\end{equation}
which depend on $\D w$ and $(\,\cdot\,,w)$ through separate integrands $f\colon\R^{N\times n}\to\R$ and $g\colon\Omega\times\R^N\to\R$, respectively. 

We focus on cases with standard assumptions on $f$, but very general hypotheses on $g$ and on the sharp dependence of the exponent $\ahe$ in the $\C^{1,\ahe}$ regularity on the latter hypotheses. More precisely, we impose on $g$ a Morrey-Hölder condition of type
\begin{equation}\label{EQCOND: g-term estimate}
  |g(x,y)-g(x,\widehat{y})|\le \zob(x) (1+|y|+|\widehat{y}|)^{\q-\b}|y-\widehat{y}|^\b
  \qquad\text{for all }x\in\Omega\text{ and }y,\widehat y\in\R^N
\end{equation}
with Hölder and growth exponents $0<\b\le\q<\infty$ for the $y$-variable and
with $x$-dependence controlled by a non-negative function $\Gamma$ whose precise
regularity will be measured on the scale of Morrey spaces. In particular, in the
general tradition of \cite{giaquinta1983differentiability} we admit
non-differentiability of $g(x,y)$ in $y$ and consequently non-availability of the Euler equation of $\mathcal{F}$.
Moreover, we directly foreshadow that the formulation of assumption \eqref{EQCOND: g-term estimate} is tailored out, to some
extent, for an application to the somewhat different setting of Massari's
$\C^{1,\alpha}$ regularity theorem \cite{massari1974esistenza,massari1975frontiere}
for sets with variational mean curvature in $\L^p$. Indeed, as we explain in more detail later on, in the Massari setting we here complement the results of
our predecessor paper \cite{SchmidtSchuettOptHoeldExp} and
achieve the final sharpening of the regularity conclusion.

However, for the moment we return to the setting of \eqref{DEF: integral functional} and aim at giving precise statements of our results, which concern interior regularity and can thus be conveniently formulated for local minimizers in the following sense:

\begin{defi}[local minimizers]\label{DEF: local-minimizer}
  Assume that $f$ is a Borel function with
  $\limsup_{|z|\to\infty}\frac{|f(z)|}{|z|^2}<\infty$ and that $g$ is a
  Carath\'eodory integrand. Then we call $u\in\W^{1,2}_\loc(\Omega,\R^N)$ with
  $g(\,\cdot\,,u)\in\L^1_\loc(\Omega)$ a local minimizer of $\mathcal{F}$ from
  \eqref{DEF: integral functional} if, for every $x\in\Omega$, there
  exists some $r>0$ with $\B_r(x)\Subset\Omega$ such that
  \[
    \int_{\B_r(x)}\big[f(\D u)+g(\,\cdot\,,u)\big]\dx
    \le\int_{\B_r(x)}\big[f(\D w)+g(\,\cdot\,,w)\big]\dx
  \]
  holds for all $w\in\W^{1,2}_u(\B_r(x),\R^N)$ with $g(\,\cdot\,,w)\in\L^1(\B_r(x))$.
\end{defi}

We remark that \eqref{EQCOND: g-term estimate} together with below-mentioned assumptions on $\beta$, $q$, $\Gamma$ ensures that
$g(\,\cdot\,,w)\in\L^1_\loc(\Omega)$ holds either for all
$w\in\W^{1,2}_\loc(\Omega,\R^N)$ or none at all. This said, it becomes apparent that the integrability requirements on the zero-order term in Definition \ref{DEF: local-minimizer} only serve to exclude trivial cases which do not allow for a meaningful notion of local minimality, and indeed when additionally imposing a very mild assumption such as $g(\,\cdot\,,0)\in\L^1_\loc(\Omega)$, for instance, we could also drop these requirements completely.

Our main result now reads as follows:

\begin{theo}[partial regularity for variational integrals with Morrey-Hölder zero-order integrand]\label{THEO: main-intro}
    We consider a $2$-strictly quasiconvex \textup{(}in the sense of the later Definition \ref{DEF: quasiconvex}\/\textup{)} $\C^2$ integrand $f\colon\R^{N\times n}\to\R$ which has at most quadratic growth in the sense of\/ $\limsup_{|z|\to\infty}\frac{|f(z)|}{|z|^2}<\infty$. Moreover, in case $n\ge3$, we abbreviate $2^\ast\coleq\frac{2n}{n{-}2}$ and assume that a Carath\'eodory integrand $g\colon\Omega\times\R^N\to\R$ satisfies \eqref{EQCOND: g-term estimate} with
    \begin{equation}\label{EQCOND: beta and q}
        \b\in{(0,1]}\,,\qquad\qquad\q\in{[\b,2^\ast)}\,,
    \end{equation}
    and with a non-negative function $\Gamma$ which in turn satisfies, for the exponents
    \[
      \sb\coleq\bigg(\frac{2^\ast}\b\bigg)'=\frac{2^\ast}{2^\ast{-}\b}
      \qquad\qquad\text{and}\qquad\qquad
      \sq\coleq\bigg(\frac{2^\ast}\q\bigg)'=\frac{2^\ast}{2^\ast{-}\q}\,,    
    \]
    the two Morrey conditions \textup{(}see Section \ref{SEC: Preliminaries} for definitions of the spaces\textup{)}
    \begin{equation}\label{EQ: Gamma-s-gamma}
      \Gamma\in\L^{\sb,n+\sb((2-\b)\ahe-\b)}_\loc(\Omega)
      \qquad\text{for some }\ahe \in (0,1)
    \end{equation}
    and
    \begin{equation}\label{EQ: Gamma-delta}
    \begin{cases}
        \zob\in\L^\sq_\loc(\Omega)\,,        &\text{if }\q>2\,,\\
        \zob\in\L^{\sq,t}_\loc(\Omega)\quad \text{with}\quad t>n-\sq\q\,,   &\text{if }\q\le2\,.
    \end{cases}
    \end{equation}    
    
    Then, for every local minimizer $u$ of\/ $\mathcal{F}$ from \eqref{DEF: integral functional}, there exists an open set $\regreg\subseteq \Omega$ such that
    \[
      |\Omega\setminus\regreg|=0
      \qquad\qquad\text{and}\qquad\qquad
      u\in\C^{1,\ahe}_\loc(\regreg,\R^N)\,.
    \]
    In case $n\in\{1,2\}$ the result remains valid with the range for $q$ in \eqref{EQCOND: beta and q} replaced by $\q\in{[\b,\infty)}$, with $\sb$ in \eqref{EQ: Gamma-s-gamma} replaced by any $s>1$, and with \eqref{EQ: Gamma-delta} replaced by the requirement $\Gamma\in\L^{\eta,\max\{n-\q,0\}}_\loc(\Omega)$ for some $\eta>1$.
\end{theo}

In Theorem \ref{THEO: main-intro}, our main point is in determining the
gradient Hölder exponent $\ahe$ of the minimizers and more precisely its sharp dependence on the parameters $\b$, $\q$, and $\Gamma$ from the Morrey-Hölder condition \eqref{EQCOND: g-term estimate} for the zero-order integrand $g$.
Somewhat surprisingly, this dependence has rarely been addressed, though even with a focus on partial $\C^{1,\ahe}$ regularity in general vectorial frameworks we are aware of an abundance of positive results in \cite{MORREY68,GIUMIR68,GIAMOD79,CAMPANATO83,ANZELLOTTI83,FUSHUT85,EVANS86,GIAMOD86,ACEFUS87,CAMPANATO87,EVAGAR87,GARIEPY89,ANZGIA88,ACEFUS89,FUSHUT91,HAMBURGER96,PASSIE97,HAMBURGER98,CARFUSMIN98,DUZGRO00,DUZGASGRO00,CARMIN01,BILFUC01a,GUIDORZI02,DUZGAS02,HAMBURGER03,DUZGASMIN04,DUZGROKRO04,DUZGROKRO05,KRIMIN05,KRIMIN05a,MINGIONE06,KRIMIN07,CHOYAN07,DUZMIN08,SCHMIDT08,SCHMIDT09A,CARLEOPASVER09,DUZMIN11,DUZMIN10A,DUZMIN10B,DIELENSTRVER12,BECK13,HOPPER16,KUSMIN16,GMEKRI19,GMEKRI23} and of counterexamples in \cite{DEGIORGI68,GIUMIR68EX,NECAS77,SVEYAN00,SVEYAN02,MUESVE03,SZEKELYHIDI04,MONSA16,MOONEY20}. However, at least the works \cite{GIAGIU84,HAMBURGER07b,SCHMIDT09B} consider conditions of types similar to \eqref{EQCOND: g-term estimate} with $\q=\beta$.
Here, \cite{GIAGIU84} establishes in fact full regularity for scalar and specific vectorial cases. In particular, for $\Gamma\in\L^\infty_\loc(\Omega)$, this work reaches $\C^{1,\ahe}$ regularity with the optimal exponent $\ahe=\frac\b{2-\b}$, which is reproduced by our result; cf.\@ situation \ref{ITEM: unif-Cbeta} below. 
Moreover, a remark in \cite{GIAGIU84} predicts --- still for $\q=\b$, the specific vectorial case mentioned, and in the context of full regularity --- even the general Hölder exponent of Theorem \ref{THEO: main-intro} (though in quite different terminology) and the one of the subsequent situation \ref{ITEM: Lp-Cbeta} with $\Gamma\in\L^p_\loc(\Omega)$. In contrast, the works \cite{HAMBURGER07b,SCHMIDT09B} deal with partial regularity in frameworks closer to ours and reach once more the exponent $\ahe=\frac\b{2-\b}$ in case $\q=\b$, $\Gamma\in\L^\infty_\loc(\Omega)$, but not the finer ones for $\Gamma$ in Morrey and $\L^p$ spaces. Concerning the exponent $\frac\b{2-\b}$, it should also be mentioned that it has in fact been discovered even before and occurs also in \cite{CRARABTAR77,PHILLIPS83A,PHILLIPS83B,ALTPHI86,HAMLUTMEAWHI07}, partially in different form and for more specific model cases. Finally, a Morrey-Hölder condition similar to \eqref{EQCOND: g-term estimate} has been treated also in \cite{SCHMIDT14}, where specifically \cite[Corollary 3]{SCHMIDT14} asserts partial $\C^{1,\ahe}$ regularity of a minimizer $u$ under a condition of type \eqref{EQCOND: g-term estimate}, possibly with $\b<\q$, plus an a priori Morrey hypothesis on $u$ itself. However, as these results are derived from a general theory of almost-minimizers, the accessible exponents $\ahe$ stay beyond $\frac\b2$ and already in case $\Gamma,|u|\in\L^\infty_\loc(\Omega)$ cannot reach the previously mentioned optimal exponent $\frac\b{2-\b}$.

At this stage, with the sharp dependence of $\ahe$ still being our main concern, we turn to several technical observations on the governing conditions \eqref{EQCOND: beta and q}, \eqref{EQ: Gamma-s-gamma}, \eqref{EQ: Gamma-delta} for the parameters and in particular on the Morrey conditions \eqref{EQ: Gamma-s-gamma} and \eqref{EQ: Gamma-delta} for $\Gamma$.

First, we put on record that only \eqref{EQ: Gamma-s-gamma} but not \eqref{EQ: Gamma-delta} influences the Hölder exponent $\ahe$. Therefore, we have conveniently formulated \eqref{EQ: Gamma-s-gamma} with the exponent
$n+\sb((2{-}\b)\ahe-\b)$ such that $\ahe\in{(0,1)}$ eventually turns out to be exactly the exponent of the resulting $\C^{1,\ahe}$ regularity. In addition, we point out that the case of interest is only $\ahe\le\frac\b{2-\b}$, since in case of larger $\ahe$ we have
$\L^{\sb,n+\sb((2{-}\b)\ahe-\b)}_\loc(\Omega)=\{0\}$, and this enforces via \eqref{EQ: Gamma-s-gamma}
and \eqref{EQCOND: g-term estimate} that $g(x,y)$ is independent of $y$ and that
the zero-order term of $\mathcal{F}$ trivializes. Turning to the complementary
condition \eqref{EQ: Gamma-delta}, we remark that this hypothesis --- though not relevant for determining $\ahe$ --- seems technically inevitable in making the first steps towards ($\C^1$) regularity in the above generality, while it can be dropped in slightly restricted frameworks to be discussed below. Let us also observe that the case distinction within \eqref{EQ: Gamma-delta} seems formally reasonable, since $q\le2$ is equivalent with $n-\sq\q\ge0$ and thus the distinction prevents us from considering negative $t$.

We remark that, similar to the modification mentioned for $n\in\{1,2\}$, also in dimension $n\ge3$ we could require \eqref{EQ: Gamma-s-gamma} in the form $\Gamma\in\L^{s,n+s((2{-}\b)\ahe-\b)}_\loc(\Omega)$ for any $s\ge\sb$. However, in view of the standard Morrey space embedding
$\L^{s,n+s((2{-}\b)\ahe-\b)}_\loc(\Omega)
\subset\L^{\sb,n+\sb((2{-}\b)\ahe-\b)}_\loc(\Omega)$ for all $s\ge\sb$ (cf.\@ Section \ref{SEC: Preliminaries}), this does not win any generality, and thus we have indeed given the above statement for $s=\sb$.

Further, we record that if one reads Theorem \ref{THEO: main-intro} for merely obtaining
$\C^{1,\ahe}$ regularity with \emph{some} $\ahe>0$, then \eqref{EQ: Gamma-s-gamma} reduces to
$\Gamma\in\L^{\sb,t}_\loc(\Omega)$ for some $t>n-\sb\b$. However, even this reduced version of
\eqref{EQ: Gamma-s-gamma} remains incomparable with \eqref{EQ: Gamma-delta} (except in the basic case $\q=\b$),
essentially since $\L^{\sq,n-\sq\q}_\loc(\Omega)$ embeds only into $\L^{\sb,n-\sb\q}_\loc(\Omega)$ but not into $\L^{\sb,n-\sb\b}_\loc(\Omega)$.

Let us also briefly point out that the conditions recorded in Theorem \ref{THEO: main-intro} for the cases $n\in\{1,2\}$ can be formally derived\footnote{Indeed, the formal replacement transforms \eqref{EQ: Gamma-delta} into either $\q>2$, $\Gamma\in\L^\eta_\loc(\Omega)$ for some $\eta>1$ or $\q\le 2$, $\Gamma\in\L^{\eta,t}_\loc(\Omega)$ for some $\eta>1$ and some positive $t>n-\eta q$. 
This is reasonable for $n=2$ and with modified case distinction between $q>1$ and $q\le1$ also for $n=1$. 
Finally, it reduces to the adapted requirement of the theorem, since in case $q\le n$ having $\Gamma\in\L^{\eta,t}_\loc(\Omega)$ for some $\eta>1$ and some positive $t>n-\eta q$ turns out to be equivalent with having $\Gamma\in\L^{\eta,n-q}_\loc(\Omega)$ for some $\eta>1$.} by replacing $2^\ast$ with an arbitrarily large finite exponent and accordingly the occurrences of $\sb$ and $\sq$ in \eqref{EQ: Gamma-s-gamma} and \eqref{EQ: Gamma-delta} with some $s>1$ and $\eta>1$ arbitrarily close to $1$. In this light, we consider the case $n\in\{1,2\}$ of Theorem \ref{THEO: main-intro} as a usual and plausible complement to the main case $n\ge3$.

Finally, we mention that the proof of Theorem \ref{THEO: main-intro} is based on the $A$-harmonic approximation method introduced in \cite{DUZSTE02} and further developed in \cite{DUZGRO00,DUZGASGRO00,DUZGAS02,DUZGROKRO04,DUZGROKRO05,DUZMIN08,SCHMIDT08,SCHMIDT09A,DIESTRVER12,DIELENSTRVER12,CRUDIE19,GMEKRI19,GMEINEDER21}, for instance.
As usual in partial regularity theory, this method also provides an explicit characterization of the regular set $\regreg$, which will be stated in the later Section \ref{SEC: regularity proofs} and may prove very handy. In particular, whenever everywhere $\C^1$ regularity or partial $\C^1$ regularity with better information on the size of $\Omega\setminus\regreg$ than just $|\Omega\setminus\regreg|=0$ is already available, then Theorem \ref{THEO: main-intro} can boost also these types of regularity to the level of $\C^{1,\ahe}$ with optimal $\ahe$. The same is true for the subsequent Theorems \ref{THEO: a-priori-bounded}, \ref{THEO: a-priori-Lipschitz}, and indeed our intended application to the situation of Massari's regularity theorem is one instance which exploits this observation.

\medskip

We find it reasonable to complement Theorem \ref{THEO: main-intro} with a variant for minimizers which are known to be in $\L^\infty_\loc(\Omega)$, since such local boundedness may be obtained a priori from maximum principles or $\C^0$ regularity methods at least in scalar cases or cases with specific structure. Anyway, under such an assumption we can fully drop the complementary Morrey condition \eqref{EQ: Gamma-delta} and impose on $\Gamma$ only the sole assumption \eqref{EQ: Gamma-s-gamma}:

\begin{theo}[partial regularity for a priori $\L^\infty_\loc$ minimizers]\label{THEO: a-priori-bounded}
  We again consider a $2$-strictly quasiconvex $\C^2$ integrand $f\colon\R^{N\times n}\to\R$ which has at most quadratic growth. Moreover, we impose on $g$ the same assumptions as in Theorem \ref{THEO: main-intro} with $q\in{[\beta,2^\ast)}$ replaced by\footnote{In fact, we could also replace the term $(1+|y|+|\widehat y|)^{\q-\b}$ in \eqref{EQCOND: g-term estimate} with an arbitrary locally bounded function of $y$ and $\widehat y$, but for simplicity we stick to the above form of \eqref{EQCOND: g-term estimate}.} $q\in{[\beta,\infty)}$ and with \eqref{EQ: Gamma-delta} dropped. Then, if a local
  minimizer $u$ of\/ $\F$ satisfies $u\in\L^\infty_\loc(\Omega,\R^N)$, the
  conclusion of Theorem \ref{THEO: main-intro} remains valid.
\end{theo}

Another standard observation in $\C^1$ regularity theory is that in the presence of even a priori \emph{gradient} bounds one merely needs locally uniform convexity/ellipticity and can in fact dispense with any global uniformity. This results in the following further variant of our results, suitable for the announced application to Massari's regularity theorem:

\begin{theo}[partial regularity for a priori $\W^{1,\infty}_\loc$ minimizers in non-uniformly elliptic cases]\label{THEO: a-priori-Lipschitz}
  We consider a $\C^2$ integrand $f\colon\R^{N\times n}\to\R$ which is strictly convex in the sense of\/ $\D^2f(z)\xi\cdot\xi>0$ for all $z,\xi\in\R^{N\times n}$ and has at most quadratic growth. Moreover, we impose on $g$ the same assumptions as in Theorem \ref{THEO: main-intro} with $q\in{[\beta,2^\ast)}$ replaced by $q\in{[\beta,\infty)}$ and with \eqref{EQ: Gamma-delta} dropped. Then, if a local minimizer $u$ of\/ $\F$ satisfies $u\in\W^{1,\infty}_\loc(\Omega,\R^N)$, the conclusion of Theorem \ref{THEO: main-intro} still holds.
\end{theo}

In order to better illustrate the significance of our assumptions \eqref{EQCOND: g-term estimate}, \eqref{EQCOND: beta and q}, \eqref{EQ: Gamma-s-gamma}, \eqref{EQ: Gamma-delta} for the zero-order integrand $g\colon\Omega\times\R^N\to\R$, we next touch briefly upon specific situations, to which we will return in full detail in the later Section \ref{SUBSEC: specific-terms}. For the moment, we tacitly understand that in case $n\in\{1,2\}$ the previously outlined modifications are in force and that the first-order integrand $f$ is suitably well-behaved. Then two model situations are as follows:

\begin{enumerate}[label={(\Alph*)}]

\item\textbf{\boldmath(uniformly) $\C^{0,\beta}$ integrands $g$:} In this basic situation, assume \eqref{EQCOND: g-term estimate}, \eqref{EQCOND: beta and q} and $\Gamma\in\L^\infty_\loc(\Omega)$. 
Then, the embedding $\L^\infty_\loc(\Omega)\subset\L^{s,n}_\loc(\Omega)$ implies \eqref{EQ: Gamma-s-gamma} and \eqref{EQ: Gamma-delta}, and our results reproduce partial $\C^{1,\ahe}$ regularity with the known exponent $\ahe=\frac\b{2-\b}$ if $\b<1$ and clearly with every $\ahe<1$ if $\b=1$.\label{ITEM: unif-Cbeta}

\item\textbf{\boldmath$\L^p$-$\C^{0,\beta}$ integrands $g$:} Now assume that \eqref{EQCOND: g-term estimate}, \eqref{EQCOND: beta and q} hold with $\Gamma\in\L^p_\loc(\Omega)$ for some $p\in{\big(\frac n\b,\infty\big)}$. 
Then, via $p\ge\sb$ and $\L^p_\loc(\Omega)\subset\L^{\sb,n-\sb n/p}_\loc(\Omega)$, we deduce \eqref{EQ: Gamma-s-gamma} with largest admissible exponent $\ahe=\frac{\b-n/p}{2-\b}$, and in case of\footnote{In case $n\ge3$ one can check by computation that this condition is an extra requirement only for $\q>2^\ast\big(1-\frac\b n\big)=2\frac{n-\b}{n-2}$.} $p\ge\sq$ we also infer \eqref{EQ: Gamma-delta}. 
Thus, our results yield partial $\C^{1,\ahe}$ regularity with $\ahe=\frac{\b-n/p}{2-\b}$ (specifically $\ahe=1{-}\frac np$ in case $\b=1$), where the additional condition $p\ge\sq$ is required for the general framework of Theorem \ref{THEO: main-intro}, but can be dropped in the a priori bounded cases of Theorems \ref{THEO: a-priori-bounded} and \ref{THEO: a-priori-Lipschitz}.\label{ITEM: Lp-Cbeta}
\end{enumerate}

In fact, a decisive motivation for the present work arises in the scalar case $N=1$ from integrands $g$ of the integral form
\begin{equation}\label{EQ: g-integral-of-H}
  g(x,y)=-\int_0^yH(x,t)\d t
  \qquad\text{for }x\in\Omega\,,y\in\R
\end{equation}
with suitable $H\in\L^1_\loc(\Omega\times\R)$. By observing that the indefinite integral of an $\L^r(\R)$ function has an $\L^r(\R)$ derivative and is $(1{-}\frac1r)$-Hölder continuous on $\R$ (or a corresponding explicit estimate with Hölder's inequality), the previous situations with $\b=\q=1-\frac1r$ then give rise to the following ones:

\begin{enumerate}[label={(\Alph*')}]
\item\textbf{\boldmath(uniformly) $\W^{1,r}$ integrands $g$:} Assume that $g$ takes the form \eqref{EQ: g-integral-of-H} with $H\in\L^\infty_\loc(\Omega,\L^r(\R))$, i.\@e.\@ $\sup_{x\in K}\|H(x,\,\cdot\,)\|_{\L^r(\R)}<\infty$ for all $K\Subset\Omega$, with some $r\in{(1,\infty]}$. We then conclude partial $\C^{1,\ahe}$ regularity with $\ahe=\frac{r-1}{r+1}$ if $r<\infty$ and clearly with every $\ahe<1$ if $r=\infty$.\label{ITEM: unif-W1r}

\item\textbf{\boldmath$\L^p$-$\W^{1,r}$ integrands $g$:} Now assume that $g$ takes the form \eqref{EQ: g-integral-of-H} with $H\in\L^p_\loc(\Omega,\L^r(\R))$, i.\@e.\@ $\int_K\|H(x,\,\cdot\,)\|_{\L^r(\R)}^p\dx<\infty$ for all $K\Subset\Omega$, with some $r\in{(1,\infty]}$ and $p\in{(nr',\infty)}$. We then conclude partial $\C^{1,\ahe}$ regularity with $\ahe=\frac{r-1-nr/p}{r+1}$ (specifically $\ahe=1{-}\frac np$ in case $r=\infty$).\label{ITEM: Lp-W1r}
\end{enumerate}

Our final application exploits Theorem \ref{THEO: a-priori-Lipschitz} in the case \ref{ITEM: Lp-W1r} with $p=r\in{(n+1,\infty)}$ and the corresponding optimal Hölder exponent $\ahe=\alpha_\mathrm{opt}\coleq\frac{p-(n+1)}{p+1}$. It concerns minimizers of Massari's functional
\[
  \mathcal{F}_H^U(F)\coleq  \P(F,U)-\int_{U\cap F}H\dx\,,
\]
defined for fixed open $U\subseteq\R^{n+1}$ and $H\in\L^1(U)$ on sets $F\subseteq\R^{n+1}$ of finite perimeter $\P(F,U)$ in $U$. In this setting, our previous work \cite{SchmidtSchuettOptHoeldExp} came up with the explicit exponent $\alpha_\mathrm{opt}$ and established partial $\C^{1,\alpha}$ regularity of minimizers of $\mathcal{F}_H^U$ with $H\in\L^p_\loc(U)$ for all $\alpha<\alpha_\mathrm{opt}$, while examples of irregularity were given for all $\alpha>\alpha_\mathrm{opt}$. Here, indeed, we finally sharpen these results by pushing regularity to the limit case $\alpha=\alpha_\mathrm{opt}$:

\begin{theo}[optimal Massari-type regularity]\label{THEO: opt-Massari}
  Consider an open set\/ $U\subseteq\R^{n+1}$ and $H\in\L^1(U)\cap\L^p_\loc(U)$ with $n+1<p<\infty$. If a set $E\subseteq\R^{n+1}$ of finite perimeter in $U$ minimizes $\mathcal{F}_{H}^U$ among all sets $F\subseteq\R^{n+1}$ of finite perimeter in $U$ such that $F\triangle E\Subset U$, then $\partial^\ast\!E\cap U$ is relatively open in $\partial E\cap U$ and is an $n$-dimensional $\C^{1,\alpha_\mathrm{opt}}$-submanifold with $\alpha_\mathrm{opt}\coleq\frac{p-(n+1)}{p+1}$. Moreover, the singular set $(\partial E\setminus\partial^\ast\! E)\cap U$ is empty in case $n\le6$ and has Hausdorff dimension at most $n{-}7$ in case $n\ge7$.
\end{theo}

The plan of the paper is now as follows. We collect preliminaries and organize our assumptions in Sections 2 and 3, respectively. Then, on the basic level, we follow the known $A$-harmonic approximation strategy: We deal with a Caccioppoli inequality in Section 4, with approximate $A$-harmonicity in Section 5, and with excess estimates in Section 6. Finally, we prove Theorems \ref{THEO: main-intro}, \ref{THEO: a-priori-bounded}, \ref{THEO: a-priori-Lipschitz} and return to the model situations \ref{ITEM: unif-Cbeta}, \ref{ITEM: Lp-Cbeta}, \ref{ITEM: unif-W1r}, \ref{ITEM: Lp-W1r} in Section 7, while in Section 8 we finally establish Theorem \ref{THEO: opt-Massari}.

\bigskip

\textbf{Acknowledgment.} We thank an anonymous referee for a careful reading of the paper and some suggestions which have led to an improved exposition.

\section{Preliminaries}\label{SEC: Preliminaries}

\subsection*{Generalities}\label{SUBSEC: Generalities}

By $\B_r(x)$ we denote the ball in $\R^{n}$ with center $x\in\R^{n}$ and radius $r>0$. 
We abbreviate $\B_r\coleq\B_r(0)$. By $\c(t_1,\dots,t_k)$, $k\in\mathbb{N}$, we denote a constant which only depends on the numbers, vectors, matrices or functions $t_1,\dots, t_k$. 
Throughout the proofs, we will neglect these dependencies and only write $\c$, which may vary from line to line.
For $p\in[1,\infty)$, we define the Sobolev exponent as $p^\ast\coleq \frac{np}{n-p}$ if $p< n$ and as (arbitrarily large) number\footnote{It may seem intuitive to set $p^\ast=\infty$ for $p>n$. However, in our context, the convention above is more useful.} $p^\ast\in{(p,\infty)}$ if $p\ge n$. 
Moreover, we denote the conjugate exponent $\frac{p}{p-1}$ of $p\in{(1,\infty)}$ by $p'$. 

For functions $u\colon\R^n\to\R^N$ and constants $\zeta\in\R^N$, $\xi\in\R^{N\times n}$, we set
\begin{equation}\label{eq:u-xi-zeta}
  u_{\xi,\zeta}(x)\coleq u(x)-\zeta-\xi x
  \qquad\text{for all }x\in \R^n\,.
\end{equation}
If $u$ is even in $\L^1_\loc(\Omega,\R^N)$, we write $(u)_\Omega$ for the mean value $\dashint_\Omega u\dx\coleq \frac{1}{|\Omega|}\int_\Omega u\dx$ and abbreviate $(u)_{x,r}$ if $\Omega=\B_r(x)$ with $x\in\R^n$ and $r>0$.
For $u\in\L^p_\loc(\Omega,\R^N)$ with $p\in{[1,\infty)}$, the Lebesgue differentiation theorem gives that a.\@e.\@ $x\in\Omega$ is an $\L^p$-Lebesgue point of $u$ in the sense that $\lim_{r\searrow0}\frac{1}{|\B_r|}\int_{\B_r(x)\cap \Omega}|u -u^\ast(x)|^p\d y=0$ holds for some corresponding Lebesgue value $u^\ast(x)\in\R^N$. Moreover, $u^\ast$ equals $u$ a.\@e.\@ and is sometimes called the Lebesgue representative of $u$.

\subsection*{Morrey spaces}\label{SUBSEC: Morrey spaces}

In this paper, we use Morrey spaces in the sense of the following definition.

\begin{definition}[Morrey spaces]
    Let $p\in[1,\infty)$ and $r\in [0,\infty)$. We define the \textbf{Morrey space}
    \[\L^{p,r}(\Omega)\coleq  \{h\in\L^p(\Omega):\|h\|_{\L^{p,r}(\Omega)}<\infty \}\,,\]
    where the value 
    \[\|h\|_{\L^{p,r}(\Omega)}\coleq \sup\limits_{\rho>0,x\in\Omega}\left(\frac{1}{\rho^r}\int_{\Omega\cap \B_\rho(x)}|h|^p\d y\right)^\frac{1}{p}\]
    is called the $\L^{p,r}$ Morrey norm of $h$.
\end{definition}

\begin{rem}\label{REM: Morrey subset relation}
 Hölder's inequality yields the optimal embedding 
    \begin{equation*}
        \L^{p,n-\alpha}(\Omega)\subseteq  \L^{r,n-a}(\Omega)\qquad \text{for all }r\le\min\left\{1,\frac{a}{\alpha}\right\}p,\text{ and } p\in[1,\infty), \,\alpha,\,a\in[0,n]\,.
    \end{equation*}
    In particular, there holds $\L^p(\Omega)=\L^{p,0}(\Omega)\subseteq  \L^{r,n-\frac{n}{p}r}(\Omega)$ for all $r\in[1,p].$ Moreover, Lebesgue's differentiation theorem implies $\L^{p,n}(\Omega)=\L^{\infty}(\Omega)$ and $\L^{p,r}(\Omega)=\{0\}$ for all $r>n$.
\end{rem}

\subsection*{Quasiconvexity and growth conditions}\label{SUBSEC: Quasiconvexity and growth}

Quasiconvexity was introduced by Morrey in \cite{morrey1952quasi} as a generalization of convexity. In several cases, quasiconvexity of the integrand turned out to be equivalent with lower semicontinuity of variational integrals. Moreover, a more recent observation is that strict quasiconvexity of the integrand is equivalent with (mean) coercivity of variational integrals; see \cite{CHEKRI17}. In view of these equivalences (strict) quasiconvexity is the decisive hypothesis in the existence theory of minimizers and in addition is also crucial for partial regularity of (local) minimizers. We recast the definitions as follows.

\begin{definition}[($2$-strict) quasiconvexity]\label{DEF: quasiconvex}
  A function $h\in\C^0(\R^{N\times n})$ is called \textbf{quasiconvex} if
  \begin{equation*}
    \dashint_{\B_1} h(\xi+\D\varphi)\dx\ge h(\xi)
  \end{equation*}
  is satisfied for all $\varphi\in \C^\infty_\cpt (\B_1,\R^N)$ and all $\xi\in\R^{N\times n}$.
  It is called {\boldmath{$2$}}\textbf{-strictly quasiconvex} if, for each bound $M>0$, there exists a positive constant\/ $\Q_M$ such that
  \begin{equation}\label{EQ: quasiconvex definition}
    \dashint_{\B_1} h(\xi+\D\varphi)\dx\ge h(\xi)+ \Q_M\dashint_{\B_1} |\D \varphi|^2 \dx
  \end{equation}
  holds for all $\varphi\in \C^\infty_\cpt (\B_1,\R^N)$ and all $\xi\in\R^{N\times n}$ with $|\xi|<M$.
\end{definition}

\begin{rem}
    Whenever a function $h$ is $2$-strictly quasiconvex and has at most quadratic growth, \eqref{EQ: quasiconvex definition} stays valid for all $\varphi\in \W^{1,2}_0(\B_1,\R^N)$ by approximation.
\end{rem}

Next we record that quasiconvexity together with the quadratic growth condition of this paper naturally implies a linear bound for the derivative of the integrand, suitable rescaled growth conditions, and the Legendre-Hadamard condition for the second derivatives of the integrand.

\begin{lemma}\label{LEM: Growth on derivative of f}
    Let $h\in\C^1(\R^n)$ be a quasiconvex function such that $|h(z)|\le \lambda(1+|z|^2)$ for some $\lambda>0$ and all $z\in\R^n$. Then there exists $\Lambda=\Lambda(n,\lambda)>0$ such that $|\D h(z)|\le\Lambda(1+|z|)$ for all $z\in\R^n$.
\end{lemma}

The proof of Lemma \ref{LEM: Growth on derivative of f} can be found, for instance, in \cite[Proposition 5.2, Lemma 5.2]{giusti2003direct}.

\begin{lemma}\label{LEM: AF inequalities}
    Let $h\in\C^2(\R^n)$ satisfy the growth condition
    \begin{equation*}
        |h(x)|\le \c_1(1+|x|^2)\qquad \text{and} \qquad  |\D h(x)|\le \c_2(1+|x|)
    \end{equation*}
    for some $\c_1,\c_2>0$ and all $x\in\R^n$. Then, for each $R>0$, there exists a positive constant $\c$ depending only on $\c_1,\c_2,R$ and $\sup_{\B_{R+1}}|\D^2 h|$ such that, for all $x\in \B_R$, $y\in\R^n$, and all $t>0$, it holds
    \begin{equation*}
        \frac{|h(x+ty)-h(x)-t\D h(x)\cdot y|}{t^2}\le \c |y|^2\qquad\text{and}\qquad \frac{|\D h(x+ty)-\D h(x)|}{t}\le \c |y|\,.
    \end{equation*}
\end{lemma}

Lemma \ref{LEM: AF inequalities} originates from \cite[Lemma II.3]{ACEFUS87} and will primarily be employed for $h$ defined on $\R^{N\times n}$. 
The proof given there relies on the mean value theorem combined with a case distinction in terms of $|ty|$.

\begin{lemma}[Legendre-Hadamard condition]\label{LEM: LH condition}
    Consider a $2$-strictly quasiconvex function $h\in\C^2(\R^{N\times n})$ and $\xi\in\R^{N\times n}$ with $|\xi|<M$ for some $M>0$. Then $\D^2 h(\xi)$ satisfies the Legendre-Hadamard condition with constant $2\Q_M$, that is, $\D^2 h(\xi)\zeta x^T\cdot\zeta x^T\ge 2\Q_M|\zeta|^2|x|^2$ for all $x\in\R^n,\,\zeta\in\R^N$.
\end{lemma} 

The proof is analogous to \cite[Proposition 5.2]{giusti2003direct}. Indeed, the second-order criterion for the minimum at $t=0$ of the single-variable function $t\mapsto\int_\Omega h(\xi+t \D \varphi)-\Q_M |t\D\varphi|^2 \dx$ gives $\int_\Omega\D^2h(\xi)\D\varphi\cdot\D\varphi\dx    \ge2\Q_M\int_\Omega|\D\varphi|^2\dx$ for all $\varphi\in\C^1_\cpt(\Omega,\R^N)$. This inequality is tested with $\varphi(y)=\eta(y)\cos(\tau x\cdot y)\zeta$ and $\varphi(y)=\eta(y)\sin(\tau x\cdot y)\zeta$, where $\eta\in\C^1_\cpt(\Omega)$ and $\tau\in\R$ are a cut-off function and a parameter, respectively, and in the limit $\tau\to\infty$ one obtains the claim. For further details we refer once more to the proof of \cite[Proposition 5.2]{giusti2003direct}.

\subsection*{\boldmath$A$-harmonic approximation}\label{SUBSEC: A-harmonic approximation}

We use the following notion of $A$-harmonic functions or in other words of weak solutions to second-order constant-coefficient linear PDE systems.

\begin{definition}[$A$-harmonic function, {{\cite[Section 3]{SCHMIDT09B}}}]
    Let $A$ be a bilinear form on $\R^{N\times n}$ satisfying
\begin{align}
    A(\zeta x^T,\zeta x^T)&\ge\lambda|\zeta|^2|x|^2\label{EQCOND: lower bound Harmonic}\\
    |A|&\le \Lambda\label{EQCOND: upper bound Harmonic}
\end{align}
for some $\lambda,\Lambda>0$ and all $x\in \R^n$, $\zeta\in \R^N$. A function $h\in \W^{1,2}(\Omega,\R^N)$ is called an {\boldmath{$A$}}\textbf{-harmonic function} on $\Omega$ if
\begin{equation*}
    \int_\Omega A(\D h,\D \varphi)\dx=0 \qquad \text{for all }\varphi\in \W^{1,2}_0(\Omega,\R^N).
\end{equation*}
\end{definition}

The next lemma can be found in a slightly more general version in \cite[Lemma 6.8]{SCHMIDT08}. It will be crucial in deriving excess estimates for local minimizers by comparison with $A$-harmonic functions and ultimately by exploiting the good estimates of linear regularity theory. 

\begin{lemma}[$A$-harmonic approximation]\label{LEM: Harmonic Approx}
Let $A$ be a bilinear form on $\R^{N\times n}$ which satisfies \eqref{EQCOND: lower bound Harmonic} and \eqref{EQCOND: upper bound Harmonic} for some $\lambda,\Lambda>0$. For every $\varepsilon>0$, there exist $\delta=\delta(\varepsilon,n,N,\lambda,\Lambda)>0$ and $\c=\c(n,N,\lambda,\Lambda)>0$ such that, whenever $w\in \W^{1,2}(\B_\rho(x_0),\R^N)$ and $\Psi\in{(0,1]}$ satisfy
\begin{equation*}
    \dashint_{\B_\rho(x_0)} |\D w|^2 \dx \le \Psi^2 \quad \text{and} \quad \left|\dashint_{\B_\rho(x_0)} A(\D w,\D \varphi)\dx\right|\le \delta\Psi\,\|\D \varphi\|_{\L^\infty(\B_\rho(x_0))}
\end{equation*}
for all $\varphi\in \W^{1,\infty}_0(\B_\rho(x_0),\R^N)$, then there exist an $A$-harmonic function $h\in\C^\infty(\B_\rho(x_0),\R^N)$ with 
\begin{equation*}
    \|\D h\|_{\C\big(\B_\frac{\rho}{2}(x_0)\big)} +\rho \|\D^2 h\|_{\C\big(\B_\frac{\rho}{2}(x_0)\big)}\le \c \quad  \text{and}\quad     \dashint_{\B_{\frac{\rho}{2}}(x_0)} \left|\frac{w-\Psi h}{\rho}\right|^2 \dx\le \Psi^2\varepsilon.
\end{equation*}
\end{lemma}

In fact, the lemma will eventually be used with the choice $A\coleq\D^2 f(\xi)$, where $f$ is the first-order integrand of our theorems and Lemma \ref{LEM: LH condition} guarantees the validity of the Legendre-Hadamard condition \eqref{EQCOND: lower bound Harmonic}.

\subsection*{An iteration lemma}\label{SUBSEC: iteration lemma}

The following lemma, for which we refer to \cite[Lemma 6.1]{giusti2003direct}, will be decisive for the proof of a later Caccioppoli inequality.

\begin{lemma}[iteration lemma]\label{LEM: Iteration lemma}
    Let $0\le r<R$ and let $v:[r,R]\to[0,\infty)$ be a bounded and non-negative function such that for all $r_1<r_2$ in $[r,R]$ the estimate 
    \begin{equation*}
        v(r_1)\le \frac{a}{(r_2-r_1)^{\tau}} +\frac{b}{(r_2-r_1)^{t}} +C + \mu v(r_2)
    \end{equation*}
    holds true with exponents $\tau,t>0$, some constants $a,b,C\ge0$ and $\mu\in[0,1)$. Then there exists a constant $\c=\c(t,\tau,\mu)>0$ such that
    \begin{equation*}
        v(r) \le \c\left(\frac{a}{(R-r)^{\tau}} +\frac{b}{(R-r)^{t}} +C \right).
    \end{equation*}
\end{lemma}

\subsection*{Campanato's integral characterization of Hölder continuity}\label{SUBSEC: Campanato}

The following proposition connects integral oscillation controls and Hölder continuity in an optimal way. It can be deduced from \cite[Theorem 2.9]{giusti2003direct}, for instance.

\begin{prop}[Campanato's integral characterization of Hölder continuity]\label{PROP: Campanatos characterisation}
    Let $p\in{[1,\infty)}$, $\alpha\in{(0,1]}$ and $v\in\L^p(\Omega)$. The Lebesgue representative $v^\ast$ of\/ $v$ is in $\C^{0,\alpha}_\loc(\Omega)$ if and only if for all $K\Subset\Omega$ there exists a constant $\c>0$ such that for all $x\in K$ and all $r>0$ with $\B_r(x)\subseteq \Omega$ the inequality 
    \begin{equation*}
        \dashint_{\B_r(x)}|v- (v)_{x,r}|^p\d y\le \c r^{\alpha p}
    \end{equation*}
    is satisfied.
\end{prop} 

\subsection*{The perimeter and variational mean curvatures}\label{SUBSEC: Perimeter and VMC}

In Section \ref{SEC: optimal Hölder exponent for Massari} we adapt our regularity results to sets of variational mean curvature in $\mathrm{L}^p$. We directly warn the reader that we set these parametric considerations in the ambient space $\R^{n+1}$ and then find the optimal exponent $\frac{p-(n+1)}{p+1}$, while in the context of the ambient space $\R^n$ the same exponent clearly reads $\frac{p-n}{p+1}$. However, our setting is convenient in transferring regularity from the non-parametric to the parametric framework.

\medskip

We now give a brief introduction to the theory of variational mean curvatures. For an $\mathcal{L}^{n+1}$-measurable set $E\subseteq  \R^{n+1}$ and an open set $U\subseteq  \R^{n+1}$, the perimeter of $E$ in $U$ is defined by 
\begin{equation*}
    \P(E,U)\coleq  \sup\left\{\int_E \mathrm{div}\varphi\dx: \varphi\in\C^1_{\mathrm{cpt}}(U,\R^{n+1}), \|\varphi\|_{\C(U)}\le 1\right\}\,.
\end{equation*}
We say that $E$ is a set of finite perimeter in $U$ if $\P(E,U)<\infty$. The perimeter measures the boundary of a set. Indeed, the structure theorem of De Giorgi guarantees $P(E,U)=\mathcal{H}^{n-1}(\partial^\ast\! E\cap U)$ for sets $E$ of locally finite perimeter in $U$, where the reduced boundary $\partial^\ast\! E$ is defined as in \cite[Definition 3.54]{ambrosio2000} and is invariant under modification of $E$ by null sets. In particular, we record $\partial E = \partial^\ast\! E$ whenever $E$ has $\C^1$ boundary.
\\
In connection with variational curvatures, we consider the measure theoretic interior 
\[E(1)\coleq \left\{x\in\R^{n+1}: \lim_{r\searrow 0}\frac{|E\cap \B_r(x)|}{|\B_r(x)|}=1\right\}\] 
as a representative of the $\mathcal{L}^n$-measurable set $E$. The main advantage is that, for every other representative $E'$ of $E$, it holds $\partial E'\supseteq \partial E(1)=\overline{\text{$\partial^\ast\!E$}}$. For more background on perimeter and $\mathrm{BV}$ theory, we refer to \cite{ambrosio2000,giusti1984,maggi2012}.

\medskip

Variational mean curvatures are motivated by the strong connection between the Massari functional $\mathcal{F}_H^U$ and the mean curvature of $\C^2$-submanifolds and in a sense indeed generalize the notion of mean curvature for (reduced) boundaries of arbitrary $\mathcal{L}^n$-measurable sets. However, the perspective differs from the classical one of differential geometry in that variational mean curvatures are indeed functions defined on the entire ambient space $\R^{n+1}$.

\begin{definition}[(local) variational mean curvatures, {{\cite[p.\@ 197]{GMT1993boundaries}}}]
Let $U\subseteq  \R^{n+1}$ be an open set. We say that a set $E\subseteq  \R^{n+1}$ of finite perimeter in $U$ has \textbf{\textup{(}local\/\textup{)} variational mean curvature \boldmath{$H\in\L^1(U)$} in \boldmath{$U$}} if 
\begin{equation*}
    \mathcal{F}_{H}^U(E)\le \mathcal{F}_{H}^U(F)\qquad \text{holds for all measurable sets }F\subseteq\R^{n+1} \text{ such that } F\sd  E\Subset U\,,
\end{equation*}
where we abbreviated
\begin{equation*}
    \mathcal{F}_H^U (F)\coleq  \P(F,U)-\int_{U\cap F}H\dx\,.
\end{equation*}
\end{definition} 

\begin{rem}[{{\cite[p. 149]{Barozzi1994}, \cite[p. 357-358]{massari1974esistenza}}}]\mbox{}\label{REM: VMCs}
\begin{enumerate}[label=\roman*)]
    \item If $H$ is a variational mean curvature of $E$ in $U$, then each $\L^1(U)$ function which is larger than $H$ on $E\cap U$ a.e. and smaller than $H$ a.e. on $U\setminus E$ is also a variational mean curvature of $E$ in $U$. In particular, from one variational mean curvature of $E$ in $U$ one can obtain infinitely many others.\label{REM: Infinitely many VMCs}
    \item If $E\subseteq \R^{n+1}$ has $\C^1$ boundary, if a part of $\partial E$ is the graph of a $\C^1$ function $v$ defined on some open $\Omega\subseteq \R^n$ such that $v(\Omega)\Subset (0,R)$ for some $R>0$, and if $H$ is a variational mean curvature of $E$ in some open $U$ such that $\overline\Omega\times{[0,R]}\subseteq U$, then $v$ minimizes the functional
    \begin{equation*}
        \G[w]\coleq  \int_\Omega \sqrt{1+|\D w(x)|^2}-\int_0^{w(x)}H(x,t)\d t\dx
    \end{equation*}
    among all $w\in \C^1_u(\Omega)$ with $w(\Omega)\Subset (0,R)$. This fact allows for transferring regularity from the non-parametric case of minimizers of \eqref{DEF: integral functional} to the parametric case of sets of variational mean curvature in $\L^p$. \label{REM: VMC implies minimizer property} 
\end{enumerate}
\end{rem}

Although Remark \ref{REM: VMCs} \ref{REM: Infinitely many VMCs} shows that variational curvatures are far from being unique, a good integrability of the curvature still implies a strong regularity result, as originally proved by Massari in \cite[Theorem 3.1, Theorem 3.2]{massari1975frontiere}.

\begin{theo}[Massari's regularity theorem]\label{THEO: Massari's regularity theorem}
Consider $p\in{(n+1,\infty]}$, $\alpha \coleq \frac{1}{4}\big(1-\frac{n+1}{p}\big)$ \textup{\big(}where we understand $\frac{1}{4}\big(1-\frac{n+1}{p}\big)=\frac{1}{4}$ in case $p=\infty$\textup{\big)}, an open set $U\subseteq\R^{n+1}$, and a set $E\subseteq \R^{n+1}$ of finite perimeter in $U$. If 
there exists a variational mean curvature $H\in\L^p(U)$ of $E$ in $U$, then the following hold.
\begin{enumerate}[label=\rm\roman*)]
    \item $\partial^\ast\! E\cap U$ is an $n$-dimensional $\C^{1,\alpha}$-manifold relatively open in $\partial E\cap U$.
    \item For all $t\in(n-7,n+1]$, it holds $\mathcal{H}^t((\partial E\setminus \partial^\ast\! E)\cap U)=0$, where we understand $\mathcal{H}^t\coleq \mathcal{H}^0$ in case $t<0$.
\end{enumerate}
\end{theo}

\begin{rem}[optimality of the parameters in Theorem \ref{THEO: Massari's regularity theorem}]\mbox{}
\begin{enumerate}[label=\roman*)]
\item
  The ranges of the parameters $p$ and $t$ are optimal. In fact, \cite[Example 2.2]{massari1994variational} and \cite[Section 2]{GMT1993boundaries} show that the theorem fails for $p\le n+1$, and in \cite[Theorem A]{bombieri1969minimal} it is proved that the Simons cone $C= \left\{x\in\R^8: x_1^2+x_2^2+x_3^2+x_4^2<x_5^2+x_6^2+x_7^2+x_8^2\right\}$ is an area-minimizing cone (i.\@e.\@ has vanishing variational mean curvature) with $\partial C\setminus\partial^\ast\!C=\{0\}$.  
\item
  As foreshadowed in the introduction, the Hölder exponent $\alpha =\frac{1}{4}\big(1-\frac{n+1}{p}\big)$ is not optimal. In fact, in our predecessor paper \cite{SchmidtSchuettOptHoeldExp} we applied Tamanini's regularity results \cite{tamaninni1984regularity} for almost-minimizers of perimeter to improve the exponent in Theorem \ref{THEO: Massari's regularity theorem} to arbitrary $\alpha<\frac{p-(n+1)}{p+1}$ and at the same time showed by counterexamples that $\C^{1,\alpha}$ regularity may fail for $\alpha>\frac{p-(n+1)}{p+1}$. In particular, this confirmed the conjecture of \cite[Remark 3.4]{massari1994variational} that Theorem \ref{THEO: Massari's regularity theorem} should hold with exponents $\alpha(n,p)$ such that $\lim_{p\to\infty}\alpha(n,p)=1$. Here, with Theorem \ref{THEO: opt-Massari} we close the last gap in this regard by showing that, for $p\in{(n+1,\infty)}$, Theorem \ref{THEO: Massari's regularity theorem} extends to the case of the limit exponent $\alpha=\alpha_{\mathrm{opt}}\coleq\frac{p-(n+1)}{p+1}$. (As a side remark, we record that in case $p=\infty$ the limit exponent $1$ and thus $\C^{1,1}$ regularity can be reached only for $n=1$, but not for $n\ge2$, as shown in \cite[Proposition 3.6]{SchmidtSchuettOptHoeldExp} and \cite[Remark 3.4]{massari1994variational}, respectively.)    
\end{enumerate}
\end{rem}

\section{Overall assumptions and settings}

We generally work with dimensions $n,N\in \mathds{N}$ and a bounded open set $ \Omega\subseteq \R^n$. As first-order integrand we fix a $\C^2$ function $f\colon\R^{N\times n}\to\R$ of at most quadratic growth
\begin{equation}\label{EQCOND: quadratic growth}
  \limsup_{|z|\to\infty}\frac{|f(z)|}{|z|^2}<\infty\,,
\end{equation}
and as zero-order integrand a Carath\'eodory function $g\colon\Omega\times \R^N\to \R$ which satisfies the Morrey-Hölder condition \eqref{EQCOND: g-term estimate} with a non-negative function $\zob$ and exponents $\b\in{(0,1]}$ and $q\in{[\beta,2^\ast)}$ as in \eqref{EQCOND: beta and q}. We continue abbreviating $\sb\coleq2^\ast/(2^\ast{-}\b)\in{(1,\infty)}$ and $\sq\coleq2^\ast/(2^\ast{-}\q)\in{(1,\infty)}$.

Next we set up conventions in order to unify the Morrey conditions \eqref{EQ: Gamma-s-gamma}, \eqref{EQ: Gamma-delta} imposed on $\zob$ for $n\ge3$ with their counterparts for $n\in\{1,2\}$. In fact, since we can replace $\Omega$ with subdomains, we may and do assume that these assumptions hold in global instead of local Morrey spaces. Furthermore, on one hand, we can assume $\ahe\le\frac\b{2-\b}$ (compare the introduction), on the other hand, for $n\in\{1,2\}$, we understand that $2^\ast\in{(2,\infty)}$ is essentially arbitrary, but large enough to ensure $\sb\le s$  for the exponent $s>1$ in Theorem \ref{THEO: main-intro} (plus a similar requirement added a few lines below). We remark that in principle this leads to a dependence of several constants in later estimates on the choice of $2^\ast$. However, whenever we explicitly list such dependencies in the sequel, we always think of the main case $n\ge3$ with $2^\ast=\frac{2n}{n-2}$ fully determined by $n$, and thus we tacitly disregard any additional dependency on the parameter $2^\ast$ in case $n\in\{1,2\}$. With these conventions in force, we may now recast \eqref{EQ: Gamma-s-gamma} for arbitrary dimension $n$ in the convenient form of
\begin{equation}\label{EQCOND: Morrey-ahe}
  \zob \in \L^{\sb,n-\sb\b+\sb(2-\b)\ahe}(\Omega)
  \text{ with }
  \ahe \in {\big(0,{\textstyle\frac{\b}{2-\b}}\big]}\cap{(0,1)}\,,
\end{equation}
and we generally abbreviate
\[
  \zob_1\coleq\|\zob\|_{\L^{\sb,n+\sb((2-\b)\ahe-\b))}( \Omega)}\,.
\]
Furthermore, in case $n\in\{1,2\}$ we require $2^\ast$ to be large enough also for $\sq\le\eta$ with the exponent $\eta>1$ of Theorem \ref{THEO: main-intro}. Then, we may recast \eqref{EQ: Gamma-delta} for arbitrary dimension $n$ as
\begin{equation}\label{EQCOND: Morrey-delta}\begin{aligned}
  \text{in case }q>\min\{2,n\}:\qquad
  &\zob\in\L^{\sq}(\Omega)\,,\\\
  \text{ in case }q\le\min\{2,n\}:\qquad
  &\zob\in\L^{\sq,n-\sq\q+2\fhe}(\Omega)
  \text{ with }\fhe\in{(0,1)}\,,
\end{aligned}\end{equation}
where e.\@g.\@ in case $n\in\{1,2\}$, $q\le n$ the passage to \eqref{EQCOND: Morrey-delta} works by observing $\zob\in\L^{\sq,n-\q}(\Omega)$ and taking $\fhe\in{(0,1)}$ such that $2\fhe\le\sq\q-\q$. We correspondingly abbreviate
\[
  \zob_2\coleq\|\zob\|_{\L^{\sq}(\Omega)}
  \qquad\qquad\text{and}\qquad\qquad
  \zob_2\coleq\|\zob\|_{\L^{\sq,n-\sq\q+2\fhe}(\Omega)}\,,
\]
respectively. Additionally, in case $q>\min\{2,n\}$ we observe $\sq\q>n$ and for later convenience still consider a fixed exponent $\fhe\in{(0,1)}$ such that
\begin{equation}\label{EQCOND: delta-add-on}
\text{in case }q>\min\{2,n\}:\qquad
  2\fhe\le\sq\q-n\,.\qquad\qquad\qquad\qquad\quad\!
\end{equation}

In handling $\D^2f$, given $t_1,\dots,t_k>0$, $k\in\mathbb{N}$, we apply the convention that $\sdb_{t_1,\dots,t_k}$ denotes a constant such that
\[
  \sdb_{t_1,\dots,t_k}\ge \sup\{|\D^2f(z)|:z\in\R^{N\times n},|z|<\c(t_1,\dots,t_k)\}\,.
\]
Moreover, since $f$ is $\C^2$, for each $M>0$, there exists a modulus of continuity $\nu_M\colon{[0,\infty)}\to{[0,\infty)}$ for $\D^2f$ such that $\lim_{t\searrow0}\nu_M(t)=\nu_M(0)=0$ and
\begin{equation}\label{EQ: modulus-of-continuity}
  |\D^2 f(A)-\D^2f(B)|\le\sqrt{\nu_M(|A-B|^2)}
  \qquad\text{for all }A,B\in\R^{N\times n}\text{ such that }|A|\le M\,,\,|B|\le M+1\,.
\end{equation}
Without loss of generality, we can choose $\nu_M$ such that it is bounded by $4\sdb_M^2$. In fact, in the sequel we mostly work with a concave upper bound $\widehat{\nu}_M$  for $\nu_M$ in the sense of the subsequent remark.

\begin{rem}\label{REM: Concave is not restrictive}
For each bounded function $\nu\colon{[0,\infty)}\to {[0,\infty)}$ with $\lim_{t\searrow0}\nu(t)=\nu(0)=0$, there exists a concave upper bound $\widehat{\nu}$ in the sense that 
    \begin{equation*}
    \widehat{\nu}\in \A_\nu \coleq \{\Theta\colon{[0,\infty)}\to{[0,\infty)}\,|\,\Theta \text{ is concave and }\Theta\ge \nu\}
    \end{equation*}
    with $\lim_{t\searrow0}\widehat{\nu}(t)=\widehat{\nu}(0)=0$ and $\sup_{[0,\infty)}\widehat{\nu}\leq\sup_{[0,\infty)}\nu$. These properties imply that $\widehat{\nu}$ is non-decreasing, and clearly one can choose $\widehat{\nu}$ as the concave hull $\inf_{\Theta\in\A_\nu}\Theta$ of $\nu$.
\end{rem}

\begin{proof}[Proof of the claim in Remark \ref{REM: Concave is not restrictive}]
    It holds $\A_\nu\neq\emptyset$ since the constant function with value $\C\coleq\sup_{[0,\infty)}\nu$ is in $\A_\nu$. Setting $\widehat\nu(t)\coleq \inf_{\Theta\in\A_\nu}\Theta(t)$ for $t\in{[0,\infty)}$, we have $\widehat\nu\in\A_\nu$ and $\widehat\nu\le\C$. For every $\varepsilon\in{(0,C]}$, there exists $\delta >0$ such that $\nu({[0,\delta]})\subseteq{[0,\varepsilon]}$. Then $\Theta^\varepsilon(t)\coleq \mathds{1}_{[\delta,\infty)}(t)\C+\mathds{1}_{[0,\delta)}(t) \left(\varepsilon+\frac{t}{\delta} (\C-\varepsilon)\right)$ is in $\A_\nu$ with $\lim_{t\searrow 0}\Theta^\varepsilon (t)=\Theta^\varepsilon(0)=\varepsilon$. Taking into account $\widehat\nu\le\Theta^\varepsilon$ and the arbitrariness of $\varepsilon$, we arrive at $\lim_{t\searrow 0}\widehat\nu(t)=\widehat\nu(0)=0$. Finally, we deduce that $\widehat\nu$ is non-decreasing: For arbitrary $t_1\le t_2$ in ${[0,\infty)}$ and $\lambda\in{(0,1]}$, we write $t_2=(1-\lambda)t_1+\lambda t_3$ for suitable $t_3\in{[0,\infty)}$ and deduce from concavity and non-negativity of $\widehat\nu$ that $\widehat\nu(t_2)\ge(1-\lambda)\widehat\nu(t_1)$. In the limit $\lambda\searrow0$ this implies $\widehat\nu(t_2)\ge\widehat\nu(t_1)$.
\end{proof}

\textbf{In all statements to follow}, we now \textbf{consider a local minimizer \boldmath$u\in\W^{1,2}(\Omega,\R^N)$} of the integral \eqref{DEF: integral functional} in the sense of Definition \ref{DEF: local-minimizer}, and we \textbf{fix a ball \boldmath$\B_\rho(x_0)\Subset \Omega$ with $\rho\in{(0,1]}$}. Beyond that, we will work in one of the following settings, which reproduce the frameworks of Theorems \ref{THEO: main-intro}, \ref{THEO: a-priori-bounded}, and \ref{THEO: a-priori-Lipschitz} in the introduction (except again for using global instead of local spaces). 

\begin{setting}\label{Setting: Quasiconvex f}
  This setting is essentially the one of Theorem \ref{THEO: main-intro} and the basic one for our purposes. We impose all assumptions previously discussed in this section and additionally assume that $f$ is $2$-strictly quasiconvex.
\end{setting}

\begin{setting}\label{Setting: release second Morrey}
  This setting is essentially the one of Theorem \ref{THEO: a-priori-bounded}. We assume that the local minimizer $u$ satisfies $u\in\L^\infty(\Omega,\R^N)$. Among the Morrey conditions discussed we keep only \eqref{EQCOND: Morrey-ahe}, while we drop \eqref{EQCOND: Morrey-delta} and \eqref{EQCOND: delta-add-on}. Instead of $\q\in{[\b,2^\ast)}$ we allow even arbitrary $\q\in{[\b,\infty)}$, and we again assume that $f$ is $2$-strictly quasiconvex.
\end{setting}

\begin{setting}\label{Setting: local convex f}
  This setting is essentially the one of Theorem \ref{THEO: a-priori-Lipschitz}. We assume that the local minimizer $u$ satisfies $u\in\W^{1,\infty}(\Omega,\R^N)$, and we define $\M\coleq  \|u\|_{\W^{1,\infty}(\Omega,\R^N)}$. Again, we keep only \eqref{EQCOND: Morrey-ahe}, but drop \eqref{EQCOND: Morrey-delta} and \eqref{EQCOND: delta-add-on}, and we allow arbitrary $\q\in{[\b,\infty)}$. Finally, we assume that, for each $M>0$, there exists $\lc_M>0$ such that we have
  \begin{equation}\label{EQCOND: Legendre}
    \D^2f(z)\xi\cdot \xi\ge \lc_M|\xi|^2
    \qquad\text{for all }\xi, z\in\R^{N\times n}
    \text{ such that }|z|\le M\,.
  \end{equation}
\end{setting}

Since \eqref{EQCOND: Legendre} implies that $f$ is convex and therefore quasiconvex (though not necessarily $2$-strictly quasiconvex in the precise sense of Definition \ref{DEF: quasiconvex}), Lemma \ref{LEM: Growth on derivative of f} and the growth condition \eqref{EQCOND: quadratic growth} imply that in each setting there exist constants $\gc>0$ and $\dgc=\dgc(n,N,\gc)>0$ such that
\begin{equation}
  |f(z)|\le \gc(1+|z|^2)\quad \text{and}\quad |\D f(z)|\le \dgc(1+|z|)
  \qquad\text{for all }z\in\R^{N\times n}.\label{EQCOND: growth condition for derivative}
\end{equation} 

\section{Caccioppoli inequality}

Caccioppoli-type inequalities are crucial in the partial regularity theory for minimizers $u$ of quasiconvex integrals. Our subsequent version essentially controls the $\L^2$ norm of $\D u_{\xi,\zeta}$ via $\L^2$ and $\L^{2^\ast}$ norms of $u_{\xi,\zeta}$, where it is crucial that the $\L^{2^\ast}$ terms are either superlinear in $u_{\xi,\zeta}$ or come with an arbitrarily small $\varepsilon$ in front.

\begin{lemma}[Caccioppoli's inequality]\label{LEM: Caccioppoli inequality} We consider Setting \ref{Setting: Quasiconvex f}. For every $M>0$ there exists a constant $\c=\c(n, N,\b,\q,\Q_M,\zob_1,\zob_2,\gc, M,\sdb_M)>0$ such that
\begin{equation*}
    \dashint_{\B_\frac{\rho}{2}(x_0)}\big|\D u_{\xi,\zeta}\big|^2 \dx
    \le \c\!\left[\dashint_{\B_\rho (x_0)}\left|\frac{u_{\xi,\zeta}}{\rho}\right|^2\dxxx+\varepsilon\left(\dashint_{\B_\rho(x_0)}\left|\frac{u_{\xi,\zeta}}{\rho}\right|^{2^\ast }\dxxx\right)^\frac{2}{2^\ast}\hspace{-.8ex}+\dashint_{\B_\rho(x_0)}\left|\frac{u_{\xi,\zeta}}{\rho}\right|^{2^\ast }\dxxx+\varepsilon^{-\frac\b{2-\b}}\rho^{2\minhe}\right]
\end{equation*}
for all $\varepsilon\in{(0,1]}$, $\zeta\in \R^N$, $\xi\in\R^{N\times n}$ with $|\zeta|+|\xi|\le M$, and the affine function $u_{\xi,\zeta}$ defined in \eqref{eq:u-xi-zeta}.
\end{lemma}

\begin{proof}
 W.\@l.\@o.\@g.\@ we assume $x_0=0$. In order to apply Lemma \ref{LEM: Iteration lemma} for $R=\rho$ and $r=\frac{\rho}{2}$, we fix $r_1<r_2$ in $\left(\frac{\rho}{2}, \rho\right)$. 
 Moreover, let $\eta\in \C^\infty_\cpt(\R^n)$ be a cut-off function with $\mathrm{spt}(\eta)\subseteq  \B_{r_2}$, $\eta= 1$ on $\B_{r_1}$ and $|\D \eta|\le \frac{2}{r_2-r_1}$. We set $v(x)\coleq u_{\xi,\zeta}(x) =  u(x)-\zeta-\xi x$, $\varphi\coleq  \eta v$ and $\psi\coleq  (1-\eta)v$. 
 By the product rule we have $\psi\in \W^{1,2}(\B_\rho,\R^N)$ and $\varphi\in \W^{1,2}_0(\B_\rho,\R^N)$ with $\varphi=v$ on $\B_{r_1}$. Employing the quasiconvexity of $f$ and the bound $|\xi|\le M$ we then estimate
\begin{equation}\label{IMPROV: boundedness Caccioppoli 1}
    \Q_M\int_{\B_{r_1}} |\D v|^2 \dx \le \Q_M\int_{\B_{r_2}} |\D \varphi|^2 \dx \le \int_{\B_{r_2}} f(\xi+\D \varphi) -f(\xi)\dx\,.
\end{equation}
Now, the equality $\D u-\xi=\D v =\D \varphi +\D \psi$ on $\B_\rho$ gives
\begin{equation}\label{IMPROV: boundedness Caccioppoli 2}
\begin{split}
    \Q_M\int_{\B_{r_1}} |\D v|^2 \dx &\le \int_{\B_{r_2}}f(\xi+\D \varphi)-f(\xi) \dx\\
    &= \F[u]-\F[u-\varphi]+ \int_{\B_{r_2}}g(x,u-\varphi)-g(x,u)\dx\\
    &\quad + \int_{\B_{r_2}}f(\D u-\D \psi)-f(\D u)\dx+ \int_{\B_{r_2}} f(\xi+\D \psi)- f(\xi)\dx\,.
\end{split}
\end{equation}
The vanishing of $\D \psi$ on $\B_{r_1}$ leads to
\begin{align*}
    \Q_M\int_{\B_{r_1}} |\D v|^2 \dx  &\le \big(\F[u]-\F[u-\varphi]\big)
    + \int_{\B_{r_2}}g(x,u-\varphi)-g(x,u)\dx\\
    &\quad +\int_{\B_{r_2}{\setminus} \B_{r_1}} f(\D u-\D \psi)-f(\D u)+ \D f(\xi)\D\psi\dx\\
    &\quad + \int_{\B_{r_2}{\setminus} \B_{r_1}} f(\xi+\D \psi)- f(\xi)- \D f(\xi)\D\psi\dx\\
    &\eqcol\mathrm{I}+\mathrm{II}+\mathrm{III}+\mathrm{IV}\,.
\end{align*}
Next we suitably estimate the right-hand side terms $\mathrm{I}$, $\mathrm{II}$, $\mathrm{III}$, and $\mathrm{IV}$. Clearly, we can control $\mathrm{I}$ by observing $\F[u]-\F[u-\varphi]\le 0$ thanks to the minimizing property of $u$. Next we handle the term $\mathrm{II}$, which for our purposes is the decisive one. To this end we recall assumption \eqref{EQCOND: g-term estimate} which gives
\begin{equation}\label{IMPROV: Caccioppoli inequality}
      \mathrm{II}=\int_{\B_{r_2}}g(x,u-\varphi)-g(x,u)\dx\le\int_{\B_{r_2}}\zob(1+|u|+|\varphi|)^{\q-\b}|\varphi|^\b\dx\le\int_{\B_{r_2}}\zob(1+|u|+|v|)^{\q-\b}|v|^\b\dx.
\end{equation}
Taking into account $|\zeta|+|\xi|\le M$ we deduce
\begin{equation}\label{EQ: g-integral estimation}
      \mathrm{II}\le \c\bigg(\int_{ \B_\rho}|v|^\b \zob\dx
    +\int_{ \B_\rho}|v|^\q \zob \dx\bigg)\,.
\end{equation}
The application of Hölder's inequality with exponents $\frac{2^\ast}\b$ and $\sb=\frac{2^\ast}{2^\ast-\b}$ gives
\begin{equation*}
    \int_{ \B_\rho}|v|^\b \zob\dx \le \left(\int_{ \B_\rho}|v|^{2^\ast} \dxx\right)^{\frac{\b}{2^\ast}}\left(\int_{ \B_\rho}\zob^{\sb}\dx\right)^{\frac{1}{\sb}}\,.
\end{equation*}

For further estimating the preceding term, we distinguish the cases $n\ge3$ and $n\in\{1,2\}$. In the case $n\ge3$, we employ Young's inequality with exponents $\frac{2}{\b}$ and $\frac2{2-\b}$ to conclude (with the arbitrary $\varepsilon>0$ from the statement of Lemma \ref{LEM: Caccioppoli inequality})
\begin{equation*}
    \int_{ \B_\rho}|v|^\b \zob\dx \le \varepsilon\left(\int_{ \B_\rho}|v|^{2^\ast} \dxx\right)^\frac{2}{2^\ast}+\varepsilon^{-\frac{\b}{2-\b}}\left(\int_{ \B_\rho} \zob^{\sb}\dx\right)^{\frac{2}{\sb(2-\b)}}\,.
\end{equation*}
Taking into account the Morrey condition $\zob\in\L^{s_\b,n-\sb\b+\sb(2-\b)\ahe}(\Omega)$ from \eqref{EQCOND: Morrey-ahe} for $\zob$ and the equality $(n-\sb\b)\frac2{\sb(2-\b)}=n$ (where the latter follows in case $n\ge3$ from the definitions of $\sb$ and $2^\ast$), we deduce
\begin{equation*}
    \int_{ \B_\rho}|v|^\b \zob\dx \le \varepsilon\left(\int_{ \B_\rho}|v|^{2^\ast} \dxx\right)^\frac{2}{2^\ast}+\zob_1^\frac2{2-\b}\varepsilon^{-\frac{\b}{2-\b}}\rho^{n+2\ahe}\,.
\end{equation*}
In view of $\frac{2}{2^\ast}n+2=n$ we further get 
\begin{equation}\label{EQ: estimate with Morrey constants}
    \int_{ \B_\rho}|v|^\b \zob\dx \le \c\Biggl[\varepsilon\rho^n\left(\dashint_{ \B_\rho}\left|\frac{v}{\rho}\right|^{2^\ast} \dxx\right)^\frac{2}{2^\ast}+\varepsilon^{-\frac{\b}{2-\b}}\rho^{n+2\ahe}\Biggr].
\end{equation}
Next we turn to the case $n\in\{1,2\}$ which requires a slight technical modification of the preceding estimates. Indeed, we still apply Young's inequality with exponents $\frac2\b$ and $\frac2{2-\b}$, but with the $\varepsilon$ in the preceding now replaced by $\varepsilon\rho^{n-\frac2{2^\ast}n-2}$. This means that we first get
\begin{equation*}
    \int_{ \B_\rho}|v|^\b \zob\dx \le \varepsilon\rho^{n-\frac2{2^\ast}n-2}\left(\int_{ \B_\rho}|v|^{2^\ast} \dxx\right)^\frac{2}{2^\ast}+\varepsilon^{-\frac{\b}{2-\b}}\rho^{\frac\b{2-\b}(\frac2{2^\ast}n+2-n)}\left(\int_{ \B_\rho} \zob^{\sb}\dx\right)^{\frac{2}{\sb(2-\b)}}\,.
\end{equation*}
Then we arrive at \eqref{EQ: estimate with Morrey constants} also in this case and essentially as before, since in spite of now $n\frac{2}{2^\ast}+2>n$ we still have $\frac\b{2-\b}(\frac2{2^\ast}n+2-n)+(n-\sb\b)\frac2{\sb(2-\b)}=n$ even with our convention of an arbitrary $2^\ast>2$ in case $n\in\{1,2\}$.

At this point we continue our reasoning back in arbitrary dimension $n$, but in estimating the second integral on the right-hand side of \eqref{EQ: g-integral estimation}, we distinguish cases for $q$. We first treat the case $q\le\min\{2,n\}$. By Young's inequality with exponents $\frac{2^\ast}{q}$ and $\sq$ and by the Morrey condition $\zob\in\L^{\sq,n-\sq\q+2\fhe}(\Omega)$ of \eqref{EQCOND: Morrey-delta}, it follows
\begin{equation}\label{EQ: q-term estimate}
    \int_{ \B_\rho}|v|^\q \zob \dx\le \int_{ \B_\rho}\left|\frac{v}{\rho}\right|^{2^\ast}\dxxx
    +\rho^{\sq\q}\int_{ \B_\rho}\zob^\sq\dx \le \int_{ \B_\rho}\left|\frac{v}{\rho}\right|^{2^\ast}\dxxx
    +\zob_2^\sq\rho^{n+2\fhe}\,.
\end{equation}
In the opposite case $q>\min\{2,n\}$, taking into account that $\zob\in\L^\sq(\Omega)$ and $\q\sq\ge n+2\fhe$ by \eqref{EQCOND: Morrey-delta} and \eqref{EQCOND: delta-add-on} and further recalling $\rho\le1$, we readily observe that the estimate \eqref{EQ: q-term estimate} is also valid.

Now, back to the general case we collect the estimates \eqref{EQ: g-integral estimation}, \eqref{EQ: estimate with Morrey constants}, \eqref{EQ: q-term estimate} and obtain
\begin{equation*}
    \mathrm{II} \le \c\Bigg[\varepsilon\rho^n\Bigg(\dashint_{\B_\rho}\bigg|\frac{v}{\rho}\bigg|^{2^\ast}\dxxx\Bigg)^{\frac{2}{2^\ast}}+\int_{\B_\rho}\bigg|\frac{v}{\rho}\bigg|^{2^\ast}\dxxx+\varepsilon^{-\frac\b{2-\b}}\rho^{n+2\min\{\ahe,\fhe\}}\Bigg]\,.
\end{equation*}

Next we turn to the term $\mathrm{III}$. Taking \eqref{EQCOND: growth condition for derivative} into account, we apply Lemma \ref{LEM: AF inequalities} to estimate
\begin{align*}
    \mathrm{III}
    &= \int_{\B_{r_2}\setminus \B_{r_1}} f(\xi +\D v-\D \psi)-f(\xi)+f(\xi)-f(\xi+\D v)+ \D f(\xi)\D\psi\dx\\
    &= \int_{\B_{r_2}\setminus \B_{r_1}} f(\xi +\D v-\D \psi)-f(\xi)-\D f(\xi)(\D v-\D\psi)- \int_0^1 \left[\D f(\xi+t\D v)-\D f(\xi)\right]\D v\d t\dx\\
    &\le \c \int_{\B_{r_2}\setminus \B_{r_1}}|\D v -\D \psi|^2+|\D v|^2 \dx\,.
\end{align*}
For the term $\mathrm{IV}$, the same type of estimate implies
\[
    \mathrm{IV}
    = \int_{\B_{r_2}\setminus \B_{r_1}} \int_0^1 \left[\D f(\xi+t\D \psi)-\D f(\xi)\right]\D\psi\d t\dx \\
    \le \c \int_{\B_{r_2}\setminus \B_{r_1}} |\D\psi|^2 \dx.
\]
The previous estimates can then be combined in order to jointly estimate $\mathrm{III}$ and $\mathrm{IV}$ by
\[
    \mathrm{III}+\mathrm{IV}
    \le \c \int_{\B_{r_2}\setminus \B_{r_1}} |\D v|^2+|\D \psi|^2 \dx
    \le \c \int_{\B_{r_2}\setminus \B_{r_1}}  |\D v|^2+\left|\frac{v}{r_2-r_1}\right|^2 \dxx\,.
\]
Collecting all estimates, we have
\begin{multline*}
        \int_{\B_{r_1}} |\D v|^2 \dx\\
        \le \c \left(\int_{\B_{r_2}\setminus \B_{r_1}}|\D v|^2\d x+ \int_{\B_\rho} \left|\frac{v}{r_2-r_1}\right|^2 \dxx+\varepsilon\rho^n\left(\dashint_{ \B_\rho}\left|\frac{v}{\rho}\right|^2 \dxx\right)^\frac{2}{2^\ast}+\int_{ \B_\rho}\left|\frac{v}{\rho}\right|^{2^\ast} \dxxx+\varepsilon^{-\frac\b{2-\b}}\rho^{n+2\minhe}\right)\,.
\end{multline*}
In order to apply Lemma \ref{LEM: Iteration lemma}, we need a constant less than $1$ in front of the first term on the right-hand side. Therefore, using Widman’s hole filling
trick, we first add $\c\int_{\B_{r_1}} |\D v|^2 \dx$ and then divide by $1+\c$ to get 
\begin{multline*}
        \int_{\B_{r_1}} |\D v|^2 \dx\\
        \le \frac{\c}{1+\c} \int_{\B_{r_2}}|\D v|^2\d x+ \int_{\B_\rho} \left|\frac{v}{r_2-r_1}\right|^2 \dxx+\varepsilon\rho^n\left(\dashint_{ \B_\rho}\left|\frac{v}{\rho}\right|^2 \dxx\right)^\frac{2}{2^\ast}+\int_{ \B_\rho}\left|\frac{v}{\rho}\right|^{2^\ast} \dxxx+\varepsilon^{-\frac\b{2-\b}}\rho^{n+2\minhe}\,.
\end{multline*}
Finally, by Lemma \ref{LEM: Iteration lemma}, we conclude
\begin{equation*}
    \dashint_{\B_\frac{\rho}{2}} |\D v|^2 \dx
    \le \c\left( \dashint_{\B_\rho} \left|\frac{v}{\rho}\right|^2\dxx+\varepsilon\left(\dashint_{ \B_\rho}\left|\frac{v}{\rho}\right|^2 \dxx\right)^\frac{2}{2^\ast}+\dashint_{\B_\rho}\left|\frac{v}{\rho}\right|^{2^\ast} \dxxx+\varepsilon^{-\frac\b{2-\b}}\rho^{n+2\minhe}\right)\,.
\end{equation*}
This completes the proof of the lemma.
\end{proof}

\section{Approximate \texorpdfstring{\boldmath$A$}{A}-harmonicity}

The following estimate prepares the ground for eventually applying the $A$-harmonic approximation of Lemma \ref{LEM: Harmonic Approx} with $A=\D^2f(\xi)$.

\begin{lemma}[approximate $A$-harmonicity]\label{LEM: approximate harmonicity} We consider Setting \ref{Setting: Quasiconvex f}. For each bound $M>0$, there exists a constant $\c=\c(n,N,\b,\q,\zob_1,\zob_2,\gc,M,\sdb_M)>0$ such that
\[
    \bigg|\dashint_{\B_\rho(x_0)}\D^2 f(\xi)\D u_{\xi,\zeta}\cdot\D\varphi\dx \bigg|
    \le \c\Bigg(\dashint_{\B_\rho(x_0)}\Bigg|\frac{u_{\xi,\zeta}}{\rho}\Bigg|^{2^\ast }\dxxx+ \sqrt{\widehat{\nu}_M(\Phi)\Phi}+\Phi+\rho^{\ahe}\Bigg)\|\D \varphi\|_{\L^\infty(\B_\rho(x_0))}
\]
for all $\varphi\in\W^{1,\infty}_0(\B_\rho(x_0),\R^N)$, $\zeta\in\R^N$, $\xi\in\R^{N\times n}$ with $|\zeta|+|\xi|\le M$,
where we use
\begin{equation*}
    \Phi\coleq  \dashint_{\B_\rho(x_0)}\big|\D u_{\xi,\zeta}\big|^2 \dxx\,,
\end{equation*}
$\widehat{\nu}_M$ from Remark \ref{REM: Concave is not restrictive}, and $u_{\xi,\zeta}$ from \eqref{eq:u-xi-zeta}.
\end{lemma}

\begin{proof}
W.\@l.\@o.\@g.\@ we assume $x_0=0$ and $\|\D \varphi\|_{\L^\infty(\B_\rho)}= 1$, which in particular implies $\|\varphi\|_{\L^\infty(\B_\rho)}\le \rho$. We abbreviate once more $v\coleq u_{\xi,\zeta}$ and --- in order to achieve convenient balance between the terms $\sigma^{\b-1}\rho^{(2-\b)\ahe}$ and $\sigma$ in the subsequent estimates \eqref{EQ: est-g-diff} and \eqref{EQ: est-Df-diff}, respectively --- choose
\[
  \sigma\coleq\rho^\ahe\in{(0,1]}\,.
\]
We then split integrals as follows (where in particular we use $\int_{\B_\rho}\D\varphi\dx=0$):
\begin{equation}\label{EQ: AA-split-terms}\begin{aligned}
  \dashint_{\B_\rho} \D^2 f(\xi)\D v\cdot \D \varphi\dx
  &=\dashint_{\B_\rho} \D^2 f(\xi)\D v\cdot \D \varphi -\D f(\D u)\cdot\D \varphi +\D f(\xi)\cdot\D \varphi \dx\\
  &\qquad+\dashint_{\B_\rho} \D f(\D u)\cdot\D \varphi +\frac{1}{\sigma}\,\big[f(\D u-\sigma\D\varphi)-f(\D u)\big] \dx\\
  &\qquad+\frac{1}{\sigma}\dashint_{\B_\rho} f(\D u) -f(\D u-\sigma\D\varphi) \dx\\
  &\eqcol\mathrm{I}+\mathrm{II}+\mathrm{III}\,.
\end{aligned}\end{equation}
We proceed by estimating the term $\mathrm{III}$ via minimality. Indeed, since $u$ is a local minimizer of $\F$, we have $\F[u]\le\F[u-\sigma \varphi]$ and may then rearrange terms and divide by $\sigma\omega_n\rho^n$ to find
\begin{equation}\label{EQ: f-terms<g-terms}
    \mathrm{III}=\frac{1}{\sigma} \dashint_{\B_\rho} f(\D u)-f(\D u-\sigma \D\varphi)\dx \le \frac{1}{\sigma}\dashint_{\B_\rho}  g(x,u-\sigma \varphi)-g(x,u) \dx\,.
\end{equation}
Assumption \eqref{EQCOND: g-term estimate} on $g$ and the bound $\|\varphi\|_{\L^\infty(\B_\rho)}\le \rho$ then imply
\begin{equation}\label{EQ: Approx A-harmonicity q-estimate}\begin{aligned}
    \mathrm{III}
    \le \frac{1}{\sigma}\dashint_{\B_\rho}\zob (1+2|u|+\sigma\rho)^{q-\b}(\sigma\rho)^\b\dx
    &\le \frac{\c}{\sigma}\dashint_{\B_\rho}\zob (1+|v|)^{q-\b}(\sigma\rho)^\b\dx\\
    &\le \c\,\sigma^{\b-1}\bigg(\rho^\b\dashint_{\B_\rho} \zob \dx+\rho^\b\dashint_{\B_\rho} \zob |v|^\q \dx\bigg)\,.
\end{aligned}\end{equation}
In further estimating the terms on the right-hand side of \eqref{EQ: Approx A-harmonicity q-estimate}, on one hand we exploit that \eqref{EQCOND: Morrey-ahe} gives $\zob\in\L^{\sb,n+\sb(2-\b)\ahe-\sb\b}(\Omega)\subseteq\L^{1,n+(2-\b)\ahe-\b}(\Omega)$ and thus
\[
  \rho^\b\dashint_{\B_\rho}\zob\dx
  \le\c\,\rho^{(2-\b)\ahe}\,.
\]
On the other hand, we re-use the estimate \eqref{EQ: q-term estimate} in form
\[
  \rho^\b\dashint_{\B_\rho}\Gamma|v|^\q\dx
  \le\rho^\b\dashint_{\B_\rho}\bigg|\frac v\rho\bigg|^{2^\ast}\dx+\c\,\rho^{\b+2\fhe}
  \le\rho^{(1-\b)\ahe}\dashint_{\B_\rho}\bigg|\frac v\rho\bigg|^{2^\ast}\dx+\c\,\rho^{(2-\b)\ahe}\,,
\]
where in the second step we exploited $\ahe\le\frac\b{2-\b}$ in discarding (at this stage insignificant) factors $\rho^{\b-(1-\b)\ahe}\le1$ and $\rho^{\b+2\fhe-(2-\b)\ahe}\le1$. Collecting the estimates and inserting the choice $\sigma=\rho^\ahe$, we arrive at
\begin{equation}\label{EQ: est-g-diff}
  \mathrm{III}
  \le \c\,\sigma^{\b-1}\left(\rho^{(1-\b)\ahe}\dashint_{\B_\rho} \left|\frac{v}{\rho}\right|^{2^\ast}\dxxx+\rho^{(2-\b)\ahe}\right)\\
  = \c\left(\dashint_{\B_\rho} \left|\frac{v}{\rho}\right|^{2^\ast}\dxxx+\rho^\ahe\right).
\end{equation}
Next, we return to the term $\mathrm{I}$ from the right-hand side of \eqref{EQ: AA-split-terms}. We partially rewrite the term by integration, then in view of $|\xi|\le M$ use \eqref{EQ: modulus-of-continuity} on the set $U\coleq\{|\D v|<1\}$ and \eqref{EQCOND: growth condition for derivative} on its complement $U^\c$, and finally apply the Cauchy-Schwarz and Jensen inequalities. In this way we find
\[\begin{aligned}
    \mathrm{I}
    &= \frac{1}{|\B_\rho|}\int_{\B_\rho\cap U} \int_0^1\left[\D^2 f(\xi)-\D^2f(\xi+\tau\D v)\right]\D v \cdot \D \varphi\d \tau\dx\\
    &\qquad+\frac{1}{|\B_\rho|}\int_{\B_\rho\cap U^\c} \D^2 f(\xi)\D v\cdot \D \varphi-\D f(\D u)\cdot\D \varphi +\D f(\xi)\cdot\D \varphi\dx\\
    &\le \dashint_{\B_\rho}\sqrt{\widehat{\nu}_M(|\D v|^2)}\,|\D v |\dx+ \frac{1}{|\B_\rho|}\int_{\B_\rho\cap U^\c} \sdb_M |\D v|+\dgc(2+|\xi|+|\D u|)\dx\\
    &\le \sqrt{\dashint_{\B_\rho}\widehat{\nu}_M(|\D v|^2) \dx}\,\,\sqrt{\dashint_{\B_\rho} |\D v |^2 \dx}+ \c\dashint_{\B_\rho} |\D v|^2 \dx\\
    &\le \sqrt{\widehat{\nu}_M(\Phi)}\sqrt{ \Phi}+ \c\,\Phi\,.
\end{aligned}\]
Similarly, in order to control the term $\mathrm{II}$, we work on the sets $\widetilde U\coleq \{|\D u|<M+1\}$ and ${\widetilde U}^\c$ and deduce
\begin{equation}\label{EQ: est-Df-diff}\begin{aligned}
    \mathrm{II}&=\dashint_{\B_\rho}\dashint_0^\sigma [\D f(\D u)- \D f(\D u-t \D\varphi)]\cdot\D \varphi\d t\dx\\
    &=\frac{1}{|\B_\rho|}\int_{\B_\rho\cap\widetilde U}\dashint_0^\sigma\int_0^1 \D^2 f(\D u-\tau t\D\varphi)\,t\D\varphi\cdot\D \varphi\d\tau\d t\dx\\
    &\qquad+\frac{1}{|\B_\rho|}\int_{\B_\rho\cap{\widetilde U}^\c}\dashint_0^\sigma [\D f(\D u)- \D f(\D u-t \D\varphi)]\cdot\D \varphi\d t\dx\\
    &\le\sdb_M\dashint_{\B_\rho}\dashint_0^\sigma t|\D \varphi|^2\d t\dx+\frac{\dgc}{|\B_\rho|}\int_{\B_\rho\cap{\widetilde U}^\c}\dashint_0^\sigma 2+|\D u| + |\D u-t \D\varphi|\d t\dx\\
    &\le\sdb_M\sigma+\c\dashint_{\B_\rho} |\D v|^2\dx\\
    &=\sdb_M\rho^\ahe+\c\,\Phi\,.
\end{aligned}\end{equation}
The combination of the estimates for $\mathrm{I}$, $\mathrm{II}$, and $\mathrm{III}$ gives
\begin{equation*}
    \dashint_{\B_\rho(x_0)}\D^2 f(\xi)\D v\cdot\D\varphi\dx \le \c\left(\dashint_{\B_\rho(x_0)}\left|\frac{v}{\rho}\right|^{2^\ast }\dxxx+\sqrt{\widehat{\nu}_M(\Phi) \Phi}+\Phi+\rho^{\ahe}\right)\,.
\end{equation*}
The corresponding lower bound can be derived similarly: The term $\mathrm{I}$ in \eqref{EQ: AA-split-terms} is kept unchanged, the terms $\mathrm{II}$ and $\mathrm{III}$ in \eqref{EQ: AA-split-terms} are replaced by
\[
  \dashint_{\B_\rho} \D f(\D u)\cdot\D \varphi -\frac{1}{\sigma}\,\big[f(\D u+\sigma\D\varphi)-f(\D u)\big] \dx
  \qquad\text{and}\qquad
  \frac{1}{\sigma}\dashint_{\B_\rho} f(\D u+\sigma\D\varphi) -f(\D u) \dx\,,
\]
respectively, and all three single terms are now estimated from below instead of above.
\end{proof}

\section{Excess estimates}\label{sec:excess}

We have now collected the main auxiliary results in order to derive estimates for the quadratic excess, defined by
\[
  \Phi(x_0,\rho)\coleq\dashint_{\B_\rho(x_0)}|\D u-(\D u)_{x_0,\rho}|^2 \dx\qquad \text{and}\qquad\Phi(\rho)\coleq \Phi(0,\rho)\,.
\]
Before turning to these estimates, however, we put on record one basic lemma on properties of the excess (which essentially follows from the observation that the quadratic function $\xi\mapsto \dashint_{\B_\rho(x_0)}|\D u-\xi|^2\dx$ has its minimum point at $(\D u)_{x_0,\rho}$).

\begin{lemma}\label{LEM: average estimate}
    For all $r\in (0,\rho]$ and all $\xi\in\R^{N\times n}$, we have the inequalities
    \begin{equation*}
        \Phi(x_0,\rho)\le \dashint_{\B_\rho(x_0)} |\D u-\xi|^2 \dx\qquad \text{and} \qquad \Phi(x_0,r)\le  \left(\frac{\rho}{r}\right)^n\Phi(x_0,\rho)\,.
    \end{equation*}
\end{lemma}

\begin{lemma}[excess improvement]\label{LEM: before excess estmate}
    We consider Setting \ref{Setting: Quasiconvex f}. For every $\kappa\in{\big(\minhe,1\big)}$ and every $M>0$, there exist constants $\theta\in{(0,1)}$, $\varepsilon>0$, $\c>0$ such that 
    \begin{equation}\label{EQ: Smallness assumption}
        \rho+\Phi(x_0,\rho)\le \varepsilon\qquad \text{and} \qquad |(\D u)_{x_0,\rho}|+|(u)_{x_0,\rho}|\le M
    \end{equation}
    together imply
    \begin{equation*}
        \Phi(x_0,\theta\rho)\le \theta^{2\kappa}\Phi(x_0,\rho)+\c\rho^{2\minhe}.
    \end{equation*}
    The constants $\c$ and $\theta$ all depend only on $n,\,N,\,\b,\,\q,\,\Q_M,\,\zob_1,\,\zob_2,\,\gc,\,M,\,\sdb_{n,N,\Q_M,\gc,M},\,\kappa$, and $\varepsilon$ additionally depends on $\ahe$ and $\nu_M$.
\end{lemma}

\begin{proof}
  We assume $x_0=0$ and $\Phi\coleq\Phi(\rho)>0$ and abbreviate
  \[
    \xi\coleq (\D u)_{0,\rho}\,,\qquad
    \zeta\coleq (u)_{0,\rho}\,,\qquad
    v(x)\coleq u_{\xi,\zeta}(x)=  u(x)-\zeta-\xi x\,,\qquad
    A\coleq\D^2f(\xi)\,.
  \]
  Our first aim is now applying the $A$-harmonic approximation lemma with this choice of $A$. To this end we first deduce from the quasiconvexity assumption on $f$, Lemma \ref{LEM: LH condition}, and the bound $|\xi|\le M$  that the conditions \eqref{EQCOND: lower bound Harmonic}, \eqref{EQCOND: upper bound Harmonic} are satisfied. Therefore, Lemma \ref{LEM: approximate harmonicity} guarantees
  \begin{equation*}
    \left|\dashint_{\B_\rho}A(\D v,\D\varphi)\dx \right|\le \c_1\left(\dashint_{\B_\rho}\left|\frac{v}{\rho}\right|^{2^\ast }\dxxx+\sqrt{\widehat{\nu}_M(\Phi)}\sqrt{\Phi} +\Phi+\rho^\ahe\right)\|\D \varphi\|_{\L^\infty(\B_\rho)}
  \end{equation*}
  for all $\varphi\in \W^{1,\infty}_0(\B_\rho,\R^N)$. Taking into account $(v)_{0,\rho}=0$ and the Sobolev-Poincar\'e inequality, we also get the estimate in form
  \begin{equation*}
    \left|\dashint_{\B_\rho}A(\D v,\D\varphi)\dx \right|\le \c_1\big(\Phi^\frac{2^\ast}{2}+\sqrt{\widehat{\nu}_M(\Phi)}\sqrt{\Phi} +\Phi+\rho^\ahe\big)\|\D \varphi\|_{\L^\infty(\B_\rho)}.
  \end{equation*}
  Now, for $\varepsilon_\mathrm{HA}\in{(0,1]}$ to be fixed below, we consider the corresponding $\delta_\mathrm{HA}=\delta_\mathrm{HA}(\varepsilon_\mathrm{HA},n,N,\Q_M,\gc)>0$ of Lemma \ref{LEM: Harmonic Approx}, and we set
  \[
    \Psi\coleq  \sqrt{\Phi}+\frac{2\c_1}{\delta_\mathrm{HA}}\rho^{\ahe}\,.
  \]
  Then we have
  \begin{equation*}
    \left|\dashint_{\B_\rho} A(\D v,\D \varphi)\dx\right|\le \Big(\c_1\big(\sqrt{\widehat{\nu}_M(\Phi)}+\sqrt{\Phi}+\Phi^\frac{2^\ast -1}{2}\big)+\frac{1}{2}\delta_\mathrm{HA} \Big)\Psi\,\|\D \varphi\|_{\L^\infty(\B_\rho)}
  \end{equation*}
  for all $\varphi\in \W^{1,\infty}_0(\B_\rho,\R^N)$. Since we can choose $\varepsilon\in(0,1)$ (which also depends on the later choice of $\varepsilon_\mathrm{HA}$) small enough for deducing from \eqref{EQ: Smallness assumption} that $\c_1 \big(\sqrt{\widehat{\nu}_M(\Phi)}+\sqrt{\Phi}+\Phi^\frac{2^\ast -1}{2}\big)\le \frac{1}{2}\delta_\mathrm{HA}$ and $\Psi\le 1$, we can ensure
  \begin{equation*}
    \left|\dashint_{\B_\rho} A(\D v,\D \varphi)\dx\right|\le \delta_\mathrm{HA}\Psi\,\|\D \varphi\|_{\L^\infty(\B_\rho)}
  \end{equation*}
  for all $\varphi\in \W^{1,\infty}_0(\B_\rho,\R^N)$ in order to then apply Lemma \ref{LEM: Harmonic Approx}. Thus, there exists an $A$-harmonic function $h\in\C^\infty(\B_\rho,\R^N)$ and a constant $\c_2>0$ such that we have
  \begin{equation}\label{EQ: Bound on the A-harmonic derivatives}
    \|\D h\|_{\C(\B_\frac{\rho}{2})} +\rho \|\D^2 h\|_{\C(\B_\frac{\rho}{2})}\le \c_2
  \end{equation}
  and 
  \begin{equation}\label{EQ: Th-L2-close-to-v}
    \dashint_{\B_{\frac{\rho}{2}}} \left|\frac{v-\Psi h}{\rho}\right|^2 \dx\le \varepsilon_\mathrm{HA}\Psi ^2\,.
  \end{equation}  
  In order to take proper advantage of \eqref{EQ: Bound on the A-harmonic derivatives} and \eqref{EQ: Th-L2-close-to-v} we first observe that \eqref{EQ: Bound on the A-harmonic derivatives} and Taylor expansion yield the bound
  \[
    \|h(x)-h(0)-\D h(0) x\|_{\C(\B_{2\theta\rho}(0))} \le 4\c_2\theta^2\rho
  \]
  for $\theta\in{\big(0,\frac14\big]}$ to be determined at the end of the proof.
  In the sequel we abbreviate
  \[
    \tilde\xi\coleq \xi+\Psi\,\D h(0)
    \qquad\text{and}\qquad
    \tilde\zeta\coleq  \zeta + \Psi\,h(0)\,,
  \]
  and we now fix $\varepsilon_\mathrm{HA}\coleq \theta^{n+4}$. Then, from \eqref{EQ: Th-L2-close-to-v} and the previous bound we obtain
  \begin{equation}\label{EQ: tilde-v-L2-small}\begin{aligned}
    \dashint_{\B_{2\theta\rho}} \bigg|\frac{u_{\tilde\xi,\tilde\zeta}}{\theta \rho}\bigg|^2 \dxx
    &\le \frac{2}{4^n\theta^{n+2}} \dashint_{\B_\frac{\rho}{2}} \left|\frac{v-\Psi h}{\rho}\right|^2 \dxx + 2\Psi ^2\sup_{x\in \B_{2\theta\rho}}\left|\frac{h(x)-h(0)-\D h(0) x}{2\theta\rho}\right|^2\\
    &\le \c \left(\frac{\varepsilon_\mathrm{HA}}{\theta^{n+2}}+\theta^2\right)\Psi ^2\\
    &\le \c \theta^2\Psi^2\,.
  \end{aligned}\end{equation}
  Furthermore, we claim that we have the auxiliary estimates
  \begin{equation}\label{EQ: tilde-xi/zeta-estimates}
    |\tilde\xi-\xi|=\Psi\,|\D h(0)|\le\c\,\Psi
    \qquad\text{and}\qquad
    |\tilde\zeta-\zeta|=\Psi\,|h(0)|\le\c\,\Psi\rho\,.
  \end{equation}
  Indeed, the first estimate in \eqref{EQ: tilde-xi/zeta-estimates} is evident from \eqref{EQ: Bound on the A-harmonic derivatives}, while the second one is now derived as follows.
  We first observe that \eqref{EQ: Bound on the A-harmonic derivatives} implies
  $|h-h(0)|\le\c\,\rho$ on $\B_{\frac{\rho}{2}}$. Then, via \eqref{EQ: Th-L2-close-to-v}, $\varepsilon_\mathrm{HA}<1$, and the Poincar\'e inequality, we conclude
  \[\begin{aligned}
    \Psi^2|h(0)|^2
    &\le\c\bigg(\dashint_{\B_\frac{\rho}{2}}|\Psi h|^2\dx+\Psi^2\rho^2\bigg)\\
    &\le\c\bigg(\dashint_{\B_\frac{\rho}{2}}|v-\Psi^2h|^2\dx+\dashint_{\B_\frac{\rho}{2}}|v|^2\dx+\Psi^2\rho^2\bigg)
    \le\c (\varepsilon_\mathrm{HA}\Psi^2\rho^2+\rho^2\Phi+\Psi^2\rho^2)
    \le\c\,\Psi^2\rho^2\,.
  \end{aligned}\]
  This completes the verification of \eqref{EQ: tilde-xi/zeta-estimates}
  and in particular ensures $|\tilde\xi|+|\tilde\zeta|\le\c$. At this stage, we are ready for the main estimations, which draw on Lemma \ref{LEM: average estimate}, on the Caccioppoli inequality of Lemma \ref{LEM: Caccioppoli inequality} with arbitrary $\varepsilon_\mathrm{CI}\in{(0,1]}$ (to be determined at the end of this proof), on the Sobolev-Poincar\'e inequality, and on the estimates \eqref{EQ: tilde-v-L2-small} and \eqref{EQ: tilde-xi/zeta-estimates}. In fact, we find
  \[\begin{aligned}
    \Phi(\theta\rho)&\le \dashint_{\B_{\theta \rho}}|\D u-\Tilde{\xi}|^2 \dx\\
    &\le \c\Bigg( \dashint_{\B_{2\theta\rho}}\bigg|\frac{u_{\tilde\xi,\tilde\zeta}}{\theta \rho}\bigg|^2 \dxx
    +\varepsilon_\mathrm{CI}\Bigg(\dashint_{\B_{2\theta\rho}}\bigg|\frac{u_{\tilde\xi,\tilde\zeta}}{\theta \rho}\bigg|^{2^\ast}\dxxx\Bigg)^\frac{2}{2^\ast}
    +\dashint_{\B_{2\theta\rho}}\bigg|\frac{u_{\tilde\xi,\tilde\zeta}}{\theta \rho}\bigg|^{2^\ast}\dxxx
    +\varepsilon_\mathrm{CI}^{-\frac\b{2-\b}}(\theta\rho)^{2\minhe}\Bigg)\\
    &\le \c\Bigg(\dashint_{\B_{2\theta\rho}}\bigg|\frac{u_{\tilde\xi,\tilde\zeta}}{\theta \rho}\bigg|^2 \dxx+\varepsilon_\mathrm{CI}\Bigg(\theta^{-n-2^\ast}\dashint_{\B_\rho}\bigg|\frac{v}{\rho}\bigg|^{2^\ast}\dxxx\Bigg)^{\frac{2}{2^\ast}}+\varepsilon_\mathrm{CI}\bigg(\frac{|\tilde\zeta-\zeta|}{\theta\rho}+|\tilde\xi-\xi|\bigg)^2\\
    &\qquad+\theta^{-n-2^\ast}\dashint_{\B_\rho}\bigg|\frac{v}{\rho}\bigg|^{2^\ast}\dxxx+\bigg(\frac{|\Tilde{\zeta}-\zeta|}{\theta\rho}+|\tilde{\xi}-\xi|\bigg)^{2^\ast}+\varepsilon_\mathrm{CI}^{-\frac\b{2-\b}}(\theta\rho)^{2\minhe}\Bigg)\\
    &\le \c\Bigg(\theta^2\Psi^2+\varepsilon_\mathrm{CI}\theta^{-\frac2{2^\ast}n-2}\Phi+\varepsilon_\mathrm{CI}\theta^{-2}\Psi^2
    +\theta^{-n-2^\ast}\Phi^\frac{2^\ast }{2}+\theta^{-2^\ast}\Psi^{2^\ast}+\varepsilon_\mathrm{CI}^{-\frac\b{2-\b}}(\theta\rho)^{2\minhe}\bigg)\,.
  \end{aligned}\]
  By using $\Phi\le\Psi^2$ and by reducing to the worst powers of $\theta$, we simplify the result of these estimations and arrive at
  \[
    \Phi(\theta\rho)
    \le\c_{3}\big[\big(\theta^2+\varepsilon_\mathrm{CI}\theta^{-\frac2{2^\ast}n-2}+\Psi^{2^\ast-2}\theta^{-n-2^\ast}\big)\Psi^2+\varepsilon_\mathrm{CI}^{-\frac\b{2-\b}}(\theta\rho)^{2\minhe}\big]\,.
  \]
  We now finalize the proof by determining the remaining parameters with the dependencies indicated in the statement. Taking into account $\kappa<1$, we first take $\theta$ small enough for having $\c_{3}\theta^2\le\frac16\theta^{2\kappa}$. Then we make $\varepsilon_\mathrm{CI}$ (which depends on the same parameters as $\theta$) small enough to ensure $\c_{3}\varepsilon_\mathrm{CI}\theta^{-\frac2{2^\ast}n-2}\le\frac16\theta^{2\kappa}$, and we decrease $\varepsilon$ such that, in addition to the earlier smallness requirement, \eqref{EQ: Smallness assumption} implies also $\c_{3}\Psi^{2^\ast-2}\theta^{-n-2^\ast}\le\frac16 \theta^{2\kappa}$. Altogether, we finally end up with
  \[
    \Phi(\theta\rho)\le\frac12\theta^{2\kappa}\Psi^2+\c_{3}\varepsilon_\mathrm{CI}^{-\frac\b{2-\b}}(\theta\rho)^{2\minhe}\le \theta^{2\kappa}\Phi+\c\rho^{2\minhe}\,,
  \]
  where we exploited that $\frac12\Psi^2\le\Phi+\frac{2\c_1}{\delta_\mathrm{HA}}\rho^{2\ahe}\le\Phi+\frac{2\c_1}{\delta_\mathrm{HA}}\rho^{2\minhe}$ by choice of $\Psi$. The proof is complete.
\end{proof}

\begin{lemma}[excess decay]\label{LEM: Excess extimate} We consider Setting \ref{Setting: Quasiconvex f}. For every $\kappa\in{\big(\minhe,1\big)}$ and every $M>0$, there exist constants $\varepsilon\in(0,1]$ and $\c>0$ such that
\begin{equation}\label{EQ: excess estimate smallness assumption}
    \rho+\Phi(x_0,\rho) \le \varepsilon\qquad \text{and} \qquad |(\D u)_{x_0,\rho}|+|(u)_{x_0,\rho}|\le M
\end{equation}
together imply
\begin{equation*}
    \Phi(x_0,r)\le \c \bigg(\bigg(\frac{r}{\rho}\bigg)^{2\kappa}\Phi(x_0,\rho)+r^{2\minhe} \bigg)
    \qquad\text{for all }r\in{(0,\rho]}\,.
\end{equation*}
Moreover, both $\varepsilon$ and $\c$ only depend on the parameters $n$, $N$, $\b$, $\q$, $\Q_M$, $\zob_1$, $\zob_2$, $\gc$, $M$, $\kappa$, and $\sdb_{n,N,\Q_M,\gc,M}$, and $\varepsilon$ additionally depends on $\ahe$ and $\nu_M$.  
\end{lemma}

\begin{proof}
  We assume $x_0=0$ and for the moment abbreviate
  \[
    \gamma\coleq\minhe\,.
  \]
  We denote by $\theta$, $\tilde\varepsilon$, and $\tilde\c$ the constants from Lemma \ref{LEM: before excess estmate} (for the given $\kappa$ and $4M$ instead of $M$). Moreover, we fix a new quantity $\varepsilon>0$ small enough that \eqref{EQ: excess estimate smallness assumption} implies
  \begin{equation}\label{EQ: smallness assumption}
    \rho+\C\rho^{2\gamma}+\Phi(\rho)\le \Tilde{\varepsilon},\quad \frac{\theta^{-\frac{n}{2}}}{1-\theta^\kappa}\sqrt{\Phi(\rho)}+\frac{\theta^{-\frac{n}{2}}}{1-\theta^\gamma}\sqrt{\C}\rho^\gamma\le M\quad \text{and}\quad \frac{\c_1\theta^{-n}}{1-\theta}\rho\left(\sqrt{\tilde{\varepsilon}}+2M\right)\le M
  \end{equation}
  with $\C\coleq  \frac{\tilde{\c}}{\theta^{2\gamma}-\theta^{2\kappa}}$ and the constant $\c_1=\c_1(n)$ of the Poincar\'e inequality relevant below. We claim that, for all integers $i\ge0$, there hold
  \begin{align}
    \Phi(\theta^i\rho)&\le \theta^{2\kappa i}\Phi(\rho)+\C(\theta^i\rho)^{2\gamma}\,,  \label{EQ: induction property 1}\\
    \theta^i\rho+\Phi(\theta^i\rho)&\le \tilde{\varepsilon}\,,         \label{EQ: induction property 2}\\
    |(\D u)_{\theta^i\rho}|&\le 2M\,,    \label{EQ: Derivative Average Estimate}\\
    |(u)_{\theta^i\rho}|&\le 2M\,.       \label{EQ: Average Estimate}
  \end{align}
  In fact, our main aim is proving \eqref{EQ: induction property 1}, but it will be convenient to establish the above set of inequalities by the following induction argument.
  
  \smallskip
  
  For $i=0$, all claims follow readily from \eqref{EQ: excess estimate smallness assumption} and the preceding choice of $\varepsilon$. So, we assume \eqref{EQ: induction property 1}--\eqref{EQ: Average Estimate} for $i=0,1,2,\ldots,k$ and deduce the validity of these claims for $i=k+1$ as well. To this end, since we have \eqref{EQ: induction property 2}--\eqref{EQ: Average Estimate} for $i=k$, we may use Lemma \ref{LEM: before excess estmate} with $\theta^k\rho$ instead of $\rho$. We find
  \begin{equation*}
    \Phi(\theta^{k+1}\rho) \le \theta^{2\kappa}\Phi(\theta^k\rho)+\tilde{\c} (\theta^k\rho)^{2\gamma}\,.
  \end{equation*}
  Then, taking into account \eqref{EQ: induction property 1} for $i=k$, we conclude
  \begin{align*}
    \Phi(\theta^{k+1}\rho)& \le \theta^{2\kappa(k+1)}\Phi(\rho)+\C (\theta^k\rho)^{2\gamma}\theta^{2\kappa}+\tilde{\c}(\theta^k\rho)^{2\gamma}\\
    & \le \theta^{2\kappa(k+1)}\Phi(\rho)+\tilde{\c} (\theta^k\rho)^{2\gamma}\left(\frac{\theta^{2\kappa}}{\theta^{2\gamma}-\theta^{2\kappa}}+1\right)\\
    &=\theta^{2\kappa(k+1)}\Phi(\rho)+\C (\theta^{k+1}\rho)^{2\gamma}\,.
  \end{align*}
  Thus, we arrive at \eqref{EQ: induction property 1} for $i=k+1$, and this implies \eqref{EQ: induction property 2} for $i=k+1$ via $\theta<1$ and the smallness assumption \eqref{EQ: smallness assumption}. Next we turn to \eqref{EQ: Derivative Average Estimate}. By the Cauchy-Schwarz inequality and \eqref{EQ: induction property 1} for $i=0,1,2,\ldots,k$, we estimate
  \[\begin{aligned}
    \left|(\D u)_{\theta^{k+1}\rho}-(\D u)_\rho\right| &\le \sum_{i=0}^k  \left|(\D u)_{\theta^{i+1}\rho}-(\D u)_{\theta^i\rho}\right|\\
    &\le \sum_{i=0}^k \left(\dashint_{\B_{\theta^{i+1}\rho}} |\D u-(\D u)_{\theta^i\rho}|^2 \dx\right)^\frac{1}{2}\\
    &\le \theta^{-\frac{n}{2}}\sum_{i=0}^k \sqrt{\Phi(\theta^i\rho)} \\
    &\le \theta^{-\frac{n}{2}}\sum_{i=0}^\infty \left(\theta^{\kappa i}\sqrt{\Phi(\rho)}+\sqrt{\C}\theta^{\gamma i}\rho^\gamma\right)\\
    &= \frac{\theta^{-\frac{n}{2}}}{1-\theta^\kappa}\sqrt{\Phi(\rho)}+\frac{\theta^{-\frac{n}{2}}}{1-\theta^\gamma}\sqrt{\C}\rho^\gamma\,.
  \end{aligned}\]
  Then, \eqref{EQ: Derivative Average Estimate} for $i=k+1$ follows via \eqref{EQ: smallness assumption} and \eqref{EQ: excess estimate smallness assumption}. For \eqref{EQ: Average Estimate}, we argue similarly, but also involve Poincar\'e's inequality. In fact, we have
  \[\begin{aligned}
    \left|(u)_{\theta^{k+1}\rho}-(u)_\rho\right| &\le \theta^{-n}\sum_{i=0}^k\dashint_{\B_{\theta^i\rho}}|u-(u)_{\theta^i\rho}|\dx \\
    &\le \c_1 \theta^{-n} \sum_{i=0}^k\theta^i\rho\dashint_{\B_{\theta^i\rho}}|\D u|\dx \\
    &\le \c_1 \theta^{-n} \sum_{i=0}^k\theta^i\rho\left(\dashint_{\B_{\theta^i\rho}}|\D u-(\D u)_{\theta^i\rho}|\dx+|(\D u)_{\theta^i\rho}|\right) \\    
    &\le \c_1 \theta^{-n} \sum_{i=0}^k\theta^i\rho\left(\sqrt{\Phi(\theta^i\rho)}+2M\right) \\
    &\le \frac{\c_1 \theta^{-n}}{1-\theta}\rho \left(\sqrt{\varepsilon}+2M\right).
  \end{aligned}\]
  Then, also \eqref{EQ: Average Estimate} for $i=k+1$ follows via \eqref{EQ: smallness assumption} and \eqref{EQ: excess estimate smallness assumption}, and the induction argument is complete.

  \smallskip
  
  Now, consider $r\in{(0,\rho]}$. There exists an integer $i\ge0$ such that $\theta^{i+1}\rho<r\le\theta^i\rho$. By \eqref{EQ: induction property 1} and Lemma \ref{LEM: average estimate}, for a constant $\c>0$, we may estimate
  \begin{align*}
    \Phi(r)&\le \theta^{-n}\Phi(\theta^i\rho)\\
    &\le \theta^{-n}\left(\theta^{2\kappa i}\Phi(\rho)+\C(\theta^i\rho)^{2\gamma}\right)\\
    &\le \c\left(\theta^{2\kappa (i+1)}\Phi(\rho)+(\theta^{i+1}\rho)^{2\gamma}\right)\\
    &\le \c\bigg(\bigg(\frac{r}{\rho}\bigg)^{2\kappa}\Phi(\rho)+
        r^{2\gamma}\bigg)\,,
  \end{align*}
  and the proof is complete.
\end{proof}

\section{Proofs of the partial regularity theorems}\label{SEC: regularity proofs}

The excess estimates prepare the ground for deducing regularity of $\D u$ on the regular set $\regreg$, defined as
\begin{equation}\label{EQ: regular set}
  \regreg
  \coleq \bigg\{x\in\Omega\,:\,\liminf_{\rho\searrow0}\dashint_{\B_\rho(x)}|\D u-(\D u)_{x,\rho}|^2\,\mathrm{d}y=0\,,\,\limsup_{\rho\searrow0} |(\D u)_{x,\rho}|+|(u)_{x,\rho}|<\infty\bigg\}\,.
\end{equation}
Specifically for $u\in\W^{1,\infty}(\Omega,\R^N)$ as in Setting \ref{Setting: local convex f}, the set $\regreg$ is nothing but the set of $\L^2$-Lebesgue points of $\D u$. However, even for arbitrary $u\in\W^{1,2}(\Omega,\R^N)$, it is a standard consequence of the Lebesgue differentiation theorem that a.\@e.\@ point belongs to $\regreg$ and in other words we have $|\Omega\setminus\regreg|=0$. Thus, this part of the conclusions in our main results is not addressed in the sequel anymore.

\subsection{Basic regularity conclusion}\label{subsec: basic regularity}

As announced above, we now apply the excess estimates of Proposition \ref{PROP: Campanatos characterisation} to establish regularity on $\regreg$. However, in the situation of Setting \ref{Setting: Quasiconvex f} we will initially reach $\C^{1,\minhe}$ regularity with exponent $\minhe$ only, while the full claim of Theorem \ref{THEO: main-intro} on $\C^{1,\ahe}$ regularity with the optimal exponent $\ahe$ is obtained only a posteriori in a further step. We now work out the last details of the initial step, while the final sharpening of the exponent is postponed to the subsequent Section \ref{subsec: sharpening}

\begin{proof}[Proof of partial $\C^{1,\minhe}$ regularity in the situation of Setting \ref{Setting: Quasiconvex f}]
  We consider an arbitrary $x_0\in \regreg$ and the exponents $\ahe,\fhe\in{(0,1)}$. 
  Then we fix the $\varepsilon>0$ from Lemma \ref{LEM: Excess extimate} which corresponds to $\kappa\coleq\frac{1+\minhe}{2}$ and $M\coleq1+\limsup_{\rho\searrow0}\big(|(\D u)_{x_0,\rho}|+|(u)_{x_0,\rho}|\big)$. 
  By the choice of $\regreg$, we have $\lim_{\rho\searrow0}\Phi(x_0,\rho)=0$. Hence, there exists $\rho_0\in{\big(0,\frac12\big]}$ with $\B_{2\rho_0}(x_0)\Subset \Omega$ and $\frac{\rho_0}{2^n}+\Phi(x_0,2\rho_0)\le \frac{\varepsilon}{2^n}$ and $|(\D u)_{x_0,\rho_0}|+|(u)_{x_0,\rho_0}|<M$. Further, by continuity of $x\mapsto |(\D u)_{x,\rho_0}|+|(u)_{x,\rho_0}|$, there exists $r_0\in{(0,\rho_0]}$ such that we have $|(\D u)_{x,\rho_0}|+|(u)_{x,\rho_0}|\le M$ for all $x\in \B_{r_0}(x_0)$. Lemma \ref{LEM: average estimate} gives 
  \[
    \Phi(x,\rho_0)
    \le \dashint_{\B_{\rho_0}(x)} |\D u-(\D u)_{x_0,2\rho_0}|^2 \dx
    \le 2^n\Phi(x_0,2\rho_0)\,,
  \]
  and thus $\rho_0+\Phi(x,\rho_0)\le \varepsilon$ holds for all $x\in \B_{r_0}(x_0)$. We may then apply Lemma \ref{LEM: Excess extimate} to infer
  \begin{equation*}
    \Phi(x,r)\le \c\bigg(\bigg(\frac{r}{\rho_0}\bigg)^{1+\minhe}\Phi(x,\rho_0)+r^{2\minhe}\bigg)
  \end{equation*}
  for all $x\in \B_{r_0}(x_0)$ and all $r\in{(0,\rho_0]}$, and we end up with
  \[
    \dashint_{\B_r(x)}|\D u-(\D u)_{x,r}|^2 \dx\le \c\left(\frac{\rho_0^{1-\minhe}}{\rho_0^{1+\minhe}}\varepsilon+1\right)r^{2\minhe}\le \c\,(\varepsilon\rho_0^{-2\minhe}+1) r^{2\minhe}
  \]
  for all $x\in\B_{r_0}(x_0)$ and $r\in{(0,\rho_0]}$. By Proposition \ref{PROP: Campanatos characterisation} with $p=2$, we find that (the Lebesgue representative of) $\D u$ is in $\C^{0,\minhe}_\loc(\B_{r_0}(x_0),\R^{N\times n})$ and therefore (the Lebesgue representative of) $u$ is in $\C^{1,\minhe}_\loc(\B_{r_0}(x_0),\R^N)$. In particular, we find $\B_{r_0}(x_0)\subseteq \regreg$, and we read off that $\regreg$ is open and that $u\in\C^{1,\minhe}_\loc (\regreg,\R^N)$ holds.
\end{proof}

\subsection[The refined Hölder exponent and a priori \texorpdfstring{$\L^\infty$}{L infinity} minimizers]{The refined Hölder exponent and a priori \texorpdfstring{\boldmath$\L^\infty$}{L infinity} minimizers}\label{subsec: sharpening}

Once we know a minimizer is $\L^\infty_{(\loc)}$ (either by previous reasoning or by assumption), we may improve the Hölder exponent from $\minhe$ to $\ahe$ and at the same time may drop the additional Morrey assumption of \eqref{EQCOND: Morrey-delta}. This follows from the subsequent adaptations of Lemma \ref{LEM: Caccioppoli inequality} and Lemma \ref{LEM: approximate harmonicity}, respectively.

\begin{lemma}[Caccioppoli inequality for a priori \texorpdfstring{$\L^\infty$}{L infinity} minimizers]\label{LEM: IMPROV: Caccioppoli inequality}
  We consider Setting \ref{Setting: release second Morrey}. For every $M>0$, there exists a constant $\c=\c\big(n,\,N,\,\b,\,\q,\,\Q_M,\,\zob_1,$ $\gc,\,M,\,\sdb_M,\|u\|_{\L^\infty(\B_\rho(x_0))}\big)$ such that we have
  \[
    \dashint_{\B_\frac{\rho}{2}(x_0)}|\D u_{\xi,\zeta}|^2 \dx \le \c\Bigg( \dashint_{\B_\rho (x_0)}\bigg|\frac{u_{\xi,\zeta}}{\rho}\bigg|^2 \dxx+\varepsilon\Bigg(\dashint_{\B_\rho(x_0)}\bigg|\frac{u_{\xi,\zeta}}{\rho}\bigg|^{2^\ast }\dxxx\Bigg)^\frac{2}{2^\ast}+\varepsilon^{-\frac\b{2-\b}}\rho^{2\ahe}\Bigg)
  \]
  for all $\varepsilon>0$, $\zeta\in \R^N$, $\xi\in\R^{N\times n}$ with $|\zeta|+|\xi|\le M$.
\end{lemma}

\begin{proof}
  The proof is analogous to the proof of Lemma \ref{LEM: Caccioppoli inequality} with the slight difference that the a priori bound ensures $|\varphi|+|u|\le \c\big(M,\|u\|_{\L^\infty(\B_\rho(x_0))}\big)$ on $\B_\rho(x_0)$. Therefore, we can improve the estimate \eqref{IMPROV: Caccioppoli inequality} in order to control the term $\mathrm{II}$ from the proof of Lemma \ref{LEM: Caccioppoli inequality} by
  \[
    \mathrm{II}
    =\int_{\B_{r_2}}g(x,u-\varphi)-g(x,u)\dx
    \le \c\int_{\B_{\rho}}\zob|v|^\b\dx\,.
  \]
  We combine this with the estimate \eqref{EQ: estimate with Morrey constants}, whose derivation is not changed, to find
  \[
    \mathrm{II}
    \le \c\Bigg(\varepsilon\rho^n\Bigg(\dashint_{\B_\rho}\bigg|\frac{v}{\rho}\bigg|^{2^\ast}\dxxx\Bigg)^\frac{2}{2^\ast} +\varepsilon^{-\frac{\b}{2-\b}}\rho^{n+2\ahe}\Bigg)\,.
  \]
  Estimating the other in exactly the same way as in the proof of Lemma \ref{LEM: Caccioppoli inequality}, we come out with the claimed form of the inequality.
\end{proof}

\begin{lemma}[approximate $A$-harmonicity for a priori $\L^\infty$ minimizers]\label{LEM: IMPROV: Approximate A-harmonicity} We consider Setting \ref{Setting: release second Morrey}. For each bound $M>0$, there exists a constant $\c=\c\big(n,N,\b,\q,\zob_1,\gc,M,\sdb_M,\|u\|_{\L^\infty(\B_\rho(x_0))}\big)$ such that we have
\begin{equation*}
    \left|\dashint_{\B_\rho(x_0)}\D^2 f(\xi)\D u_{\xi,\zeta}\cdot\D\varphi\dx \right|\le \c\left(\sqrt{\widehat{\nu}_M(\Phi)\Phi}+\Phi+\rho^{\ahe}\right)\|\D \varphi\|_{\L^\infty(\B_\rho(x_0))}
\end{equation*}
for all $\varphi\in\W^{1,\infty}_0(\B_\rho(x_0),\R^N)$, $\zeta\in\R^n$, $\xi\in\R^{N\times n}$ with $|\xi|<M$, where $\Phi$ is defined as in Lemma \ref{LEM: approximate harmonicity}.
\end{lemma}

\begin{proof}
  The proof is analogous to the proof of Lemma \ref{LEM: approximate harmonicity} with the estimates from \eqref{EQ: Approx A-harmonicity q-estimate} to \eqref{EQ: est-g-diff} replaced by simply
  \[
    \mathrm{III}
    \le\c\sigma^{\b-1}\rho^\b\dashint_{\B_\rho}\zob\dx
    \le\c\sigma^{\b-1}\rho^{(2-\b)\ahe}
    =\c\rho^\ahe
  \]
  with $\c$ depending also on $\|u\|_{\L^\infty(\B_\rho)}$. The estimates for the other terms remain the same as in the proof of Lemma \ref{LEM: approximate harmonicity}.
\end{proof}

At this stage, we finalize the proof of the first two results from the introduction.

\begin{proof}[Proof of Theorem \ref{THEO: main-intro} and Theorem \ref{THEO: a-priori-bounded}]
  We first work in the situation of Setting \ref{Setting: release second Morrey} and thus prove Theorem \ref{THEO: a-priori-bounded}. Indeed, with Lemmas \ref{LEM: IMPROV: Caccioppoli inequality} and \ref{LEM: IMPROV: Approximate A-harmonicity} at hand we may go over the proofs of Lemma \ref{LEM: before excess estmate} and Lemma \ref{LEM: Excess extimate} in order to reach the same conclusions with $\minhe$ replaced by $\ahe$ (and with constants which now depend on $\|u\|_{\L^\infty(\B_\rho(x_0))}$, but no longer on $\zob_2$). Then, in analogy with the reasoning of Section \ref{subsec: basic regularity} we arrive at the regularity claim $u\in\C^{1,\ahe}_\loc(\regreg,\R^N)$.

  \smallskip
    
  Now we turn to Setting \ref{Setting: Quasiconvex f} and finalize the proof of Theorem \ref{THEO: main-intro}. To this end, we observe that the regularity proved in Section \ref{subsec: basic regularity} implies in particular $u\in\L^\infty_\loc(\regreg,\R^N)$. Hence, on every open $\widetilde\Omega\Subset\regreg$ we are back precisely to the situation of Setting \ref{Setting: release second Morrey}, and $u\in\C^{1,\ahe}_\loc(\regreg,\R^N)$ is available from the previous reasoning.
\end{proof}

\subsection[Non-uniform ellipticity and a priori \texorpdfstring{$\W^{1,\infty}$}{W\^{}\{1,\textbackslash{}infty\}} minimizers]{Non-uniform ellipticity and a priori \texorpdfstring{\boldmath$\W^{1,\infty}$}{W\^{}\{1,\textbackslash{}infty\}} minimizers}\label{SUBSEC: a priori Lipschitz minimizers}

Once we know a minimizer is $\W^{1,\infty}_{(\loc)}$, we can even deal with \emph{locally} uniform ellipticity in the sense of merely \eqref{EQCOND: Legendre}. This standard observation leads to Theorem \ref{THEO: a-priori-Lipschitz} and eventually also helps in proving Theorem \ref{THEO: opt-Massari}. On the technical side, it rests on yet another slightly adapted Caccioppoli inequality, as stated next.

\begin{lemma}[Caccioppoli inequality for a priori \texorpdfstring{$\W^{1,\infty}$}{W 1 infinity} minimizers in non-uniformly elliptic cases]\label{LEM: Caccioppoli inequality without quasiconvexity}
  We consider Setting \ref{Setting: local convex f}. For every $M>0$, there exists a constant $\c=\c(n,N,\b,\q,\zob_1,\gc,M,\M,\sdb_M,\lc_{M+\M})$ such that we have
  \[
    \dashint_{\B_\frac{\rho}{2}(x_0)}|\D u_{\xi,\zeta}|^2 \dx \le \c\Bigg( \dashint_{\B_\rho (x_0)}\bigg|\frac{u_{\xi,\zeta}}{\rho}\bigg|^2 \dxx+\varepsilon\Bigg(\dashint_{\B_\rho(x_0)}\bigg|\frac{u_{\xi,\zeta}}{\rho}\bigg|^{2^\ast }\dxxx\Bigg)^\frac{2}{2^\ast}+\varepsilon^{-\frac\b{2-\b}}\rho^{2\ahe}\Bigg)
  \]
  for all $\varepsilon>0$, $\zeta\in \R^N$, $\xi\in\R^{N\times n}$ with $|\zeta|+|\xi|\le M$ and all $x_0\in \Omega$, $\rho\in(0,1]$ such that $\B_\rho (x_0)\Subset \Omega$ and $\M\coleq  \|u\|_{\W^{1,\infty}(\B_\rho(x_0))}$.
\end{lemma}

\begin{proof}
  Since $u$ is in particular in $\L^\infty(\Omega,\R^N)$, we can basically repeat the proof of Lemma \ref{LEM: Caccioppoli inequality} with the modifications of Lemma \ref{LEM: IMPROV: Caccioppoli inequality}. However, in this reasoning we replace the estimate \eqref{IMPROV: boundedness Caccioppoli 1} based on quasiconvexity of $f$ with the following computation. Indeed, we first recall $\D v=\D u_{\xi,\zeta}=\D u-\xi$ and record $|\xi+\tau\D v|\le |\xi|+|\D u|\le M+\M$ for all $\tau\in{[0,1]}$ in the present situation. On the basis of this observation, we then use the possibly non-uniform ellipticity \eqref{EQCOND: Legendre} to estimate and rewrite
  \[\begin{aligned}
    \frac{\lc_{M+\M}}{2}\int_{\B_{r_1}} |\D v|^2 \dx
    &\le \int_{\B_{r_1}} \int_0^1 \int_0^1 \D^2f(\xi+\tau t\D v)\D v\cdot \D v \, t\d \tau\d t \dx\\
    &\le \int_{\B_{r_2}} \int_0^1 \int_0^1 \D^2f(\xi+\tau t\D \varphi)\D \varphi\cdot \D \varphi \, t\d \tau\d t \dx\\
    &= \int_{\B_{r_2}}f(\xi+\D \varphi)-f(\xi)- \D f(\xi)\cdot\D\varphi\dx\\
    &=  \int_{\B_{r_2}}f(\xi+\D \varphi)-f(\xi) \dx\,.
  \end{aligned}\]
  As foreshadowed above, the resulting inequality is then used as a one-to-one substitute for \eqref{IMPROV: boundedness Caccioppoli 1} in the arguments already used for Lemma \ref{LEM: Caccioppoli inequality} and Lemma \ref{LEM: IMPROV: Caccioppoli inequality}.
\end{proof}

\begin{proof}[Proof of Theorem \ref{THEO: a-priori-Lipschitz}]
  We rely on Lemma \ref{LEM: Caccioppoli inequality without quasiconvexity} and on Lemma \ref{LEM: IMPROV: Approximate A-harmonicity} and otherwise on the same arguments explicated in Sections \ref{sec:excess}, \ref{subsec: basic regularity}, and \ref{subsec: sharpening} (where several constants now depend on $\M\coleq\|u\|_{\W^{1,\infty}(\B_\rho(x_0),\R^N)}$ and $\lc_{M+\M}$, but no longer on $\Q_M$, and where in the proof of Lemma \ref{LEM: before excess estmate} the requirement \eqref{EQCOND: lower bound Harmonic} comes directly from \eqref{EQCOND: Legendre}). By these arguments we then arrive at $u\in\C^{1,\ahe}_\loc(\regreg,\R^N)$ as before.
\end{proof}

\subsection[\texorpdfstring{$\L^p$}{L\^{}p}-Hölder zero-order terms and \texorpdfstring{$\L^p$}{L\^{}p}-\texorpdfstring{$\W^{1,r}$}{W\^{}\{1,r\}} zero-order terms]{\texorpdfstring{\boldmath$\L^p$}{L\^{}p}-Hölder zero-order terms and \texorpdfstring{\boldmath$\L^p$}{L\^{}p}-\texorpdfstring{\boldmath$\W^{1,r}$}{W\^{}\{1,r\}} zero-order terms}\label{SUBSEC: specific-terms}

In view of the embeddings of $\L^p$ spaces into Morrey spaces (cf.\@ Remark \ref{REM: Morrey subset relation}), our main results have straightforward corollaries for the case of an $\L^p$-Hölder zero-order term. Basically, this issue has already been touched upon in situations \ref{ITEM: unif-Cbeta} and \ref{ITEM: Lp-Cbeta} of the introduction, but still we prefer to explicate it here:

\begin{cor}[partial regularity for variational integrals with $\L^p$-Hölder zero-order integrand]\label{COR: Hölder-0-order-terms}
  We suppose that $f$ is as Setting \ref{Setting: Quasiconvex f}, Setting \ref{Setting: release second Morrey}, or Setting \ref{Setting: local convex f}. For the Carath\'eodory integrand $g\colon\Omega\times\R^N\to\R$ we assume \eqref{EQCOND: g-term estimate} with merely
  \[
    \b\in{(0,1]}\,,\qquad
    \q\in{[\b,2^\ast)}\,,\qquad
    \Gamma\in\L^p_\loc(\Omega)
    \text{ for some }p\in \bigg(\frac{n}{\b},\infty\bigg]
  \]
  and in case of Setting \ref{Setting: Quasiconvex f} additionally with $p\ge\sq$ \textup{(}where, for $n\in\{1,2\}$, the range for $q$ should be read as $q\in{[\b,\infty)}$ and the additional requirement as void\/\textup{)}. Then, for every local minimizer $u$ of\/ $\mathcal{F}$ from \eqref{DEF: integral functional}, we have
  \[
    u\in\C^{1,\ahe}_\loc(\regreg,\R^N)
    \qquad\text{with}\qquad
    \begin{cases}
      \displaystyle\ahe=\frac{\b-n/p}{2-\b}&\text{in case }p<\infty\,,\\[1.5ex]
      \displaystyle\ahe=\frac\b{2-\b}&\text{in case }p=\infty\,,\,\b<1\,,\\
      \displaystyle\text{any }\ahe<1&\text{in case }p=\infty\,,\,\b=1\,.
    \end{cases}
  \]
\end{cor}

\begin{proof}
  We first treat the case $p<\infty$. Since in case $n\ge3$ we have $p>\frac n\b>\sb$ and in case $n\in\{1,2\}$ can deduce $p\ge\sb$ from $p>\frac n\b\ge1$ by taking $2^\ast$ large enough, in any dimension we infer $\zob\in\L^p(\Omega)\subseteq\L^{\sb,n-\sb\frac np}(\Omega)=\L^{\sb,n+\sb((2-\b)\ahe-\b)}(\Omega)$ for $\ahe=\frac{\b-n/p}{2-\b}$. In case of Setting \ref{Setting: Quasiconvex f}, by assumption or choice of $2^\ast$ we additionally get $p\ge\sq$ and $\zob\in\L^p(\Omega)\subseteq\L^{\sq,n-\sq\frac np}(\Omega)$ with $n-\sq\frac np>n-\sq\b\ge n-\sq\q$. Thus, we may apply Theorem \ref{THEO: main-intro}, Theorem \ref{THEO: a-priori-bounded}, or Theorem \ref{THEO: a-priori-Lipschitz}, respectively, to deduce the claimed regularity.

  \smallskip
  
  The case $p=\infty$ is similar. Since we have $\zob\in\L^\infty(\Omega)\subseteq\L^{s,n}(\Omega)$ for all $s\in[1,\infty)$, the claimed regularity comes from Theorem \ref{THEO: main-intro}, Theorem \ref{THEO: a-priori-bounded}, or Theorem \ref{THEO: a-priori-Lipschitz}, respectively, with the choice $\ahe=\frac{\b}{2-\b}$ in case $\b<1$ and with arbitrary $\ahe\in{(0,1)}$ in case $\beta=1$.
\end{proof}

We also explicate a corresponding result for an $\L^p$-$\W^{1,r}$ zero-order term or, in other words, for a zero-order term of certain integral form. This essentially reproduces situations \ref{ITEM: unif-W1r} and \ref{ITEM: Lp-W1r} of the introduction.

\begin{cor}[partial regularity for scalar variational integrals with \texorpdfstring{$\L^p$}{L\^{}p}-\texorpdfstring{$\W^{1,r}$}{W\^{}\{1,r\}} zero-order integrands]\label{COR: Integral LrLp type terms}
  We suppose that $f$ is as Setting \ref{Setting: Quasiconvex f}, Setting \ref{Setting: release second Morrey}, or Setting \ref{Setting: local convex f} in the scalar case $N=1$. Moreover, we consider $H\in\L^1_\loc(\Omega\times\R)$ which satisfies
  \[
    H\in\L^p(\Omega,\L^r(\R))
    \qquad\text{with }r\in{(1,\infty]}\text{ and }p\in{(nr',\infty]}
  \]
  \textup{\big(}with integrability understood in the sense of\/
  $\int_\Omega\|H(x,\,\cdot\,)\|_{\L^r(\R)}^p\dx<\infty$\textup{\big)}. Then, for every local minimizer $u\in\W^{1,2}_\loc(\Omega)$ of the variational integral
  \begin{equation}\label{EQ: non-parametric-Massari-functional}
    \F[w]\coleq\int_\Omega\bigg[f(\D w(x))-\int_0^{w(x)}H(x,t)\,\d t\bigg]\dx\,,
  \end{equation}
  we have
  \[
    u\in\C^{1,\ahe}_\loc(\regreg)
    \qquad\text{with}\qquad
    \begin{cases}
      \displaystyle\ahe=\frac{r-1-nr/p}{r+1}&\text{in case }p<\infty\,,\,r<\infty\,,\\[1ex]
      \displaystyle\ahe=1-n/p&\text{in case }p<\infty\,,\,r=\infty\,,\\[.5ex]
      \displaystyle\ahe=\frac{r-1}{r+1}&\text{in case }p=\infty\,,\,r<\infty\,,\\
      \displaystyle\text{any }\ahe<1&\text{in case }p=\infty\,,\,r=\infty\,.
    \end{cases}
  \]
\end{cor}

\begin{proof}
  We set
  \[
    g(x,y)\coleq -\int_0^y H(x,t)\,\d t
    \qquad\text{for }x\in\Omega\text{ and }y\in\R
  \]
  and then observe that Hölder's inequality yields the $\L^p$-Hölder condition
  \[
    |g(x,y)-g(x,\widehat y)|
    \le\Gamma(x)|y-\widehat y|^\b  
    \qquad\text{for all }x\in\Omega\text{ and }y,\widehat y\in\R
  \]
  with $\zob\in\L^p(\Omega)$ given by $\zob(x)\coleq\|H(x,\,\cdot\,)\|_{\L^r(\R)}$ and with $\b\coleq\frac1{r'}\in{(0,1]}$. Since we assume $p>nr'=\frac n\b$, this brings us in position to apply Corollary \ref{COR: Hölder-0-order-terms} with $\q=\b$ (where now in view of $\sq=\sb$ and the initial reasoning in the previous proof the additional requirement $p\ge\sq$ is always valid). By inserting $\b=\frac1{r'}=\frac{r-1}r$ into the exponents $\ahe$ of Corollary \ref{COR: Hölder-0-order-terms}, the regularity outcome then takes the form of the current claim.
\end{proof}

\begin{rem}
  At least in case $p=r<\infty$, the Hölder exponent $\alpha$ reached in Corollary \ref{COR: Integral LrLp type terms} is optimal. This can be confirmed by transferring the counterexamples of \cite[Section 4]{SchmidtSchuettOptHoeldExp} from the parametric to the non-parametric setting (compare also Section \ref{SEC: optimal Hölder exponent for Massari} for transferring regularity the opposite way round).
\end{rem}

\section{The optimal Hölder exponent in Massari's regularity theorem}\label{SEC: optimal Hölder exponent for Massari}

At this stage, we recall that our final aim is improving Massari's regularity theorem up to the limit Hölder exponent $\alpha_\mathrm{opt}$. This will be achieved by considering the non-parametric Massari-type functional \eqref{EQ: non-parametric-Massari-functional} with $H\in\L^p(\Omega\times\R)=\L^p(\Omega,\L^p(\R))$ and on applying Corollary \ref{COR: Integral LrLp type terms} in the particular case $p=r$. We already record that in this case the requirement $p>nr'$ means nothing but $p>n+1$ and that, for $p=r<\infty$, the Hölder exponent reached in Corollary \ref{COR: Integral LrLp type terms} actually boils down to our target exponent $\alpha_\mathrm{opt}=\frac{p-(n+1)}{p+1}$. In fact, with Corollary \ref{COR: Hölder-0-order-terms} at hand, the proof of Theorem \ref{THEO: opt-Massari} essentially reduces to deducing a suitable non-parametric minimality property from the parametric one assumed:

\begin{proof}[Proof of Theorem \ref{THEO: opt-Massari}]
  We consider a set $E\subseteq\R^{n+1}$ of finite perimeter in an open set $U\subseteq \R^{n+1}$ and a variational mean curvature $H\in\L^p(U)$ of $E$ in $U$. Since we may replace $H$ by $H^+\mathds{1}_E-H^-\mathds{1}_{E^\c}$ according to Remark \ref{REM: VMCs} \ref{REM: Infinitely many VMCs}, we can assume $H\ge0$ on $E$ and $H\le0$ on $E^\c$. We now derive the optimal regularity near an arbitrary fixed $x_0\in\partial^\ast\!E\cap U$.
  
  \smallskip
  
  \emph{Step 1\textup{:} Non-parametric rewriting with smooth variations.} By isometry invariance of perimeter and variational mean curvatures and by Theorem \ref{THEO: Massari's regularity theorem}, we may directly assume that $\partial E\cap \Omega=\partial^*\!E\cap\Omega$ is $\C^1$ and $E\cap(\overline{ \Omega} \times {[0,R]})$ is the subgraph of $u\in \mathrm{C}^1(\overline{ \Omega})$ for a bounded open set $ \Omega\subseteq  \R^{n}$ and a constant $R>0$ such that $u(\overline{ \Omega})\subseteq  (0,R)$ and $x_0\in\Omega\times{(0,R)}\Subset U$. In this situation, Remark \ref{REM: VMCs} \ref{REM: VMC implies minimizer property} asserts that $u$ minimizes 
  \begin{equation*}
    \G[w]
    \coleq \int_\Omega \bigg[ \sqrt{1+|\D w(x)|^2}-\int_0^{w(x)}H(x,t)\d t \bigg] \dx
  \end{equation*}
  among all functions $w=u+\varphi$ with $\varphi\in\C^1_\cpt( \Omega)$ such that $w(\overline{ \Omega})\subseteq  (0,R)$.
    
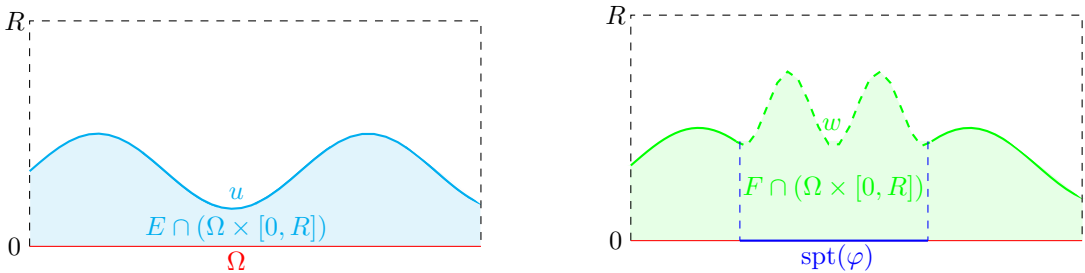
\begin{figure}[h]
        \centering
        \begin{tikzpicture}

\filldraw[white, fill=cyan!10] (0,0)--(0,1)--(6,1)--(6,0);

\filldraw[thick,cyan, fill=cyan!10, variable=\x,domain=0:2.5,samples=20] 
  plot ({\x},{0.5*sin(100*\x)+1});
  
\filldraw[thick,cyan, fill=white, variable=\x,domain=1.5:4,samples=20] 
  plot ({\x},{0.5*sin(100*\x)+1});

\filldraw[thick,cyan, fill=cyan!10, variable=\x,domain=3:6,samples=20] 
  plot ({\x},{0.5*sin(100*\x)+1});

\filldraw[thick,cyan, fill=white, variable=\x,domain=5:6,samples=20] 
  plot ({\x},{0.5*sin(100*\x)+1});

\filldraw[white] (5,1.35)--(6,0.567)--(6,1.37);

\draw[dashed]   (0,0)--(0,3)--(6,3)--(6,0);
\draw[red]      (0,0)--(6,0);

\node[cyan] at (2.75,0.24) {$E\cap (\Omega\times[0,R])$};
\node[cyan] at (2.75,0.7) {$u$};

\node at (-0.2,0) {$0$};
\node at (-0.2,3) {$R$};
\node[red] at (2.75,-0.2) {$\Omega$};

\end{tikzpicture}
\qquad\qquad 
\begin{tikzpicture}

\filldraw[green!10] (1.4,1.321)--(4,1.321)--(6,0.567)--(0,1)--(1.4,1.321); 

\filldraw[green!10] (0,0)--(0,1)--(6,1)--(6,0);

\filldraw[thick,green, fill=green!10, variable=\x,domain=0:1.4,samples=20] 
  plot ({\x},{0.5*sin(100*\x)+1});

\filldraw[thick,green, fill=green!10, variable=\x,domain=4:6,samples=20] 
  plot ({\x},{0.5*sin(100*\x)+1});

\filldraw[white] (5,1.35)--(6,0.567)--(6,1.37);

\filldraw[thick, green, fill=green!10, dashed, variable= \x, domain=1.4:4, samples=20] 
plot ({\x},{-0.5*sin(300*\x)+1.75});

\draw[dashed]   (0,0)--(0,3)--(6,3)--(6,0);
\draw[red]      (0,0)--(6,0);

\node[green] at (2.7,0.7) {$F\cap (\Omega\times[0,R])$};
\node[green] at (2.675,1.55) {$w$};

\node at (-0.2,0) {$0$};
\node at (-0.2,3) {$R$};
\draw[blue,dashed] (1.45,1.321)--(1.45,0);
\draw[blue,dashed] (3.95,1.321)--(3.95,0);
\draw[blue, thick] (1.45,0)--(3.95,0);
\node[blue] at (2.75,-0.22) {$\mathrm{spt}(\varphi)$};

\end{tikzpicture}
        \caption{Local parameterization of $E$ via $u$ and of test sets $F$ via $w$}
    \end{figure}
    
  \emph{Step 2\textup{:} Extension to $\W^{1,2}$ variations with image in ${[\varepsilon,R-\varepsilon]}$}. 
  We first argue briefly that we can approximate $\psi\in \W^{1,2}_0( \Omega)$ with $\varepsilon\le u+\psi\le R-\varepsilon$ a.\@e.\@ on $\Omega$ for some fixed $\varepsilon>0$ by $\varphi_k\in\C^1_\cpt(\Omega)$ such that $(u+\varphi_k)(\overline{ \Omega})\subseteq{(0,R)}$ with respect to the $\W^{1,2}$ norm and a.\@e.\@ on $ \Omega$. Indeed, possibly decreasing $\varepsilon$ we may assume $u(\overline{ \Omega})\subseteq{(\varepsilon,R-\varepsilon)}$, and we approximate $\psi$ by $\psi_k\in\C^1_\cpt(\Omega)$ in $\W^{1,2}( \Omega)$ and a.\@e.\@ on $\Omega$. In order to truncate suitably, we choose $T_\varepsilon\in\C^1\big(\R,{\big[\frac13\varepsilon,R-\frac13\varepsilon\big]}\big)$ with bounded derivative $(T_\varepsilon)'$ such that $T_\varepsilon(y)=y$ for $y\in{\big[\frac23\varepsilon,R-\frac23\varepsilon\big]}$. It is then a standard matter to check that, for $k\to\infty$, we have convergence of $T_\varepsilon(u+\psi_k)$ to $u+\psi$ in $\W^{1,2}(\Omega)$ and a.\@e.\@ on $\Omega$ and  consequently also of $\varphi_k\coleq T_\varepsilon(u+\psi_k)-u\in\C^1_\cpt(\Omega)$ to $\psi$ in the same sense. Taking into account that $(u+\varphi_k)(\overline{ \Omega})\subseteq{\big[\frac13\varepsilon,R-\frac13\varepsilon\big]}$ by choice of $T_\varepsilon$, we have established the approximation claim.

  \smallskip

  Now, the convergence in $\W^{1,2}(\Omega)$ and a.\@e.\@ on $\Omega$ implies convergence of the first-order and zero-order terms of the functional $\G$, respectively (where, for the latter one, we make use of the dominated convergence theorem). All in all, we infer convergence of $\G[u+\varphi_k]$ to $\G[u+\psi]$. Hence, we can conclude that $u$ minimizes $\G$ among all functions in $u+\psi\in\W^{1,2}_u(\Omega)$ with $\varepsilon\le u+\psi\le R-\varepsilon$ a.\@e.\@ on $\Omega$ for some fixed $\varepsilon>0$.

  \smallskip
    
  \emph{Step 3\textup{:} Extension to arbitrary $\W^{1,2}$ variations}.
  Finally, we pass on to fully arbitrary competitors without need for the $(\varepsilon,R-\varepsilon)$ requirement of Step 2. We consider an arbitrary $\Phi \in \W_u^{1,2}( \Omega)$ and, for $\varepsilon>0$ small enough that $u(\overline{ \Omega})\subseteq{(\varepsilon,R-\varepsilon)}$, record that the sharp truncation $\Psi\coleq\max\{\min\{\Phi,R-\varepsilon\},\varepsilon\}$ is in $ \W^{1,2}_u( \Omega)$ with $\D \Psi= \mathds{1}_{(\varepsilon,R-\varepsilon)}(\Phi)\, \D \Phi$ a.e. on $ \Omega$. We further define $\Tilde{H}(x,t)\coleq \mathds{1}_{[0,R]}(t) H(x,t)$. Then, since $\tilde{H}\in \L^1( \Omega\times \R)\cap\L^p( \Omega\times \R)$ satisfies $H\ge0$ on $ \Omega\times[0,\varepsilon]\subseteq  E$ and $H\le0$ on $ \Omega\times{[R-\varepsilon,R]}\subseteq  E^\c$, it follows
  \[\begin{aligned}
    \Tilde{\G}[\Phi]&\coleq  \int_ \Omega \bigg[ \sqrt{1+|\D \Phi(x)|^2}-\int_0^{\Phi(x)}\Tilde{H}(x,t)\d t \bigg] \dx \\
    &\ge \int_ \Omega \bigg[ \sqrt{1+|\D \Psi(x)|^2}-\int_0^{\Psi(x)}\Tilde{H}(x,t)\d t \bigg] \dx
    = \G[\Psi] \ge \G[u] = \Tilde{\G}[u]\,,
  \end{aligned}\]
  where the last rewritings are based on the minimality established in Step 2 and on $u(\overline{ \Omega})\subseteq{(0,R)}$, respectively.

  \smallskip
    
  \emph{Step 4\textup{:} Application of the non-parametric regularity result}.
  The first-order integrand $f(z)\coleq \sqrt{1+|z|^2}$ of the functional $\widetilde \G$ is smooth and of at most quadratic growth (in fact, linear growth) with
  \[
    \D^2 f(z)\xi\cdot \xi\ge \frac{|\xi|^2}{(1+M^2)^\frac{3}{2}}
    \qquad\text{for all }z,\xi\in\R^{n}\text{ with }|z|\le M\,.
  \]
  Thus, $f$ is as in Setting \ref{Setting: local convex f} with $N=1$, while $\tilde H$ satisfies the conditions of Corollary \ref{COR: Integral LrLp type terms} with $r=p\in{(n+1,\infty)}$, and moreover we already know $u\in\C^1(\overline{\Omega})$ and thus read off $\regreg=\Omega$ from \eqref{EQ: regular set}. All in all, as indicated at the beginning of this section, we may then apply Corollary \ref{COR: Integral LrLp type terms} to deduce that the minimizer $u$ of $\tilde\G$ is in $\C^{1,\alpha}_\loc(\Omega)$ with exactly the claimed exponent $\alpha=\alpha_\mathrm{opt}=\frac{p-(n+1)}{p+1}$. Since we worked near an arbitrary $x_0\in\partial^{\ast}\!E\cap U$, we have verified that $\partial^{\ast}\!E\cap U$ is an $n$-dimensional $\C^{1,\alpha_\mathrm{opt}}$-submanifold (and relatively open in $\partial E\cap U$). Finally, the assertions on the size of the singular set $(\partial E\setminus\partial^{\ast}\!E)\cap U$ are already contained in Theorem \ref{THEO: Massari's regularity theorem}.
\end{proof}

\phantomsection


\addcontentsline{toc}{section}{References}
\begin{thebibliography}{55}
\bibitem{ACEFUS87}
{\sc E.~Acerbi and N.~Fusco}, {{A regularity theorem for minimizers of
  quasiconvex integrals}}, {\it Arch. Ration. Mech. Anal.} {\bf99} (1987),
  261--281.

\bibitem{ACEFUS89}
{\sc E.~Acerbi and N.~Fusco}, {{Local regularity
  for minimizers of non convex integrals}}, {\it Ann. Sc. Norm. Super. Pisa,
  Cl. Sci., IV. Ser.} {\bf16} (1989), 603--636.

\bibitem{ALTPHI86}
{\sc H.~Alt and D.~Phillips}, {{A free boundary problem for semilinear
  elliptic equations}}, {\it J. Reine Angew. Math.} {\bf368} (1986),
  63--107.

\bibitem{ambrosio2000}
{\sc L.~Ambrosio, N.~Fusco, and D.~Pallara}, {{\it Functions of Bounded
  Variation and Free Discontinuity Problems}}, Oxford University Press, 2000.

\bibitem{ANZELLOTTI83}
{\sc G.~Anzellotti}, {{On the $\mathrm{C}^{1,\alpha}$-regularity of
  $\omega$-minima of quadratic functionals}}, {\it Boll. Unione Mat. Ital., VI.
  Ser., C, Anal. Funz. Appl.} {\bf2} (1983), 195--212.

\bibitem{ANZGIA88}
{\sc G.~Anzellotti and M.~Giaquinta}, {{Convex functionals and partial
  regularity}}, {\it Arch. Ration. Mech. Anal.} {\bf102} (1988), 243--272.

\bibitem{Barozzi1994}
{\sc E.~Barozzi}, {{The curvature of a set with finite area}}, {\it Atti
  Accad. Naz. Lincei, Cl. Sci. Fis. Mat. Nat., IX. Ser., Rend. Lincei, Mat.
  Appl.} {\bf5} (1994), 149--159.

\bibitem{BECK13}
{\sc L.~Beck}, {{Regularity versus singularity for elliptic problems in two
  dimensions}}, {\it Adv. Calc. Var.} {\bf6} (2013), 415--432.

\bibitem{BILFUC01a}
{\sc M.~Bildhauer and M.~Fuchs}, {{Partial regularity for variational
  integrals with $(s,\mu,q)$-growth}}, {\it Calc. Var. Partial Differ. Equ.}
  {\bf13} (2001), 537--560.

\bibitem{bombieri1969minimal}
{\sc E.~Bombieri, E.~De~Giorgi, and E.~Giusti}, {{Minimal cones and the
  Bernstein problem}}, {\it Invent. Math.} {\bf7} (1969), 243--268.

\bibitem{CAMPANATO83}
{\sc S.~Campanato}, {{H{\"o}lder continuity of the solutions of some
  nonlinear elliptic systems}}, {\it Adv. Math.} {\bf48} (1983), 16--43.

\bibitem{CAMPANATO87}
{\sc S.~Campanato}, {{Elliptic systems
  with non-linearity $q$ greater or equal to two. Regularity of the solution of
  the Dirichlet problem}}, {\it Ann. Mat. Pura Appl., IV. Ser.} {\bf147}
  (1987), 117--150.

\bibitem{CARFUSMIN98}
{\sc M.~Carozza, N.~Fusco, and G.~Mingione}, {{Partial regularity of
  minimizers of quasiconvex integrals with subquadratic growth}}, {\it Ann.
  Mat. Pura Appl., IV. Ser.} {\bf175} (1998), 141--164.

\bibitem{CARLEOPASVER09}
{\sc M.~Carozza, C.~Leone, A.~{Passarelli di Napoli}, and A.~Verde}, {\em
  {Partial regularity for polyconvex functionals depending on the {Hessian}
  determinant}}, {\it Calc. Var. Partial Differ. Equ.} {\bf35} (2009),
  215--238.

\bibitem{CARMIN01}
{\sc M.~Carozza and G.~Mingione}, {{Partial regularity of minimizers of
  quasiconvex integrals with subquadratic growth: the general case}}, {\it Ann.
  Pol. Math.} {\bf77} (2001), 219--243.

\bibitem{CHEKRI17}
{\sc C.Y.~Chen and J.~Kristensen}, {{On coercive variational integrals}}, {\it
  Nonlinear Anal., Theory Methods Appl., Ser. A, Theory Methods} {\bf153}
  (2017), 213--229.

\bibitem{CHOYAN07}
{\sc S.~Cho and X.~Yan}, {{On the singular set for Lipschitzian critical
  points of polyconvex functionals}}, {\it J. Math. Anal. Appl.} {\bf336}
  (2007), 372--398.

\bibitem{CRARABTAR77}
{\sc M.G.~Crandall, P.H.~Rabinowitz, and L.~Tartar}, {{On a Dirichlet
  problem with a singular nonlinearity}}, {\it Commun. Partial Differ.
  Equations} {\bf2} (1977), 193--222.

\bibitem{CRUDIE19}
{\sc D.~{Cruz-Uribe}, L.~Diening}, {{Sharp $\mathcal{A}$-harmonic
  approximation}}, {\it Appl. Anal.} {\bf98} (2019), 374--380.

\bibitem{DEGIORGI68}
{\sc E.~{De Giorgi}}, {{Un esempio di estremali discontinue per un problema
  variazionale di tipo ellittico}}, {\it Boll. Unione Mat. Ital., IV. Ser.}
  {\bf1} (1968), 135--137.

\bibitem{DIESTRVER12}
{\sc L.~Diening, B.~Stroffolini, A.~Verde}, {{The $\varphi$-harmonic
  approximation and the regularity of $\varphi$-harmonic maps}}, {\it J.
  Differ. Equations} {\bf253} (2012), 1943--1958.

\bibitem{DIELENSTRVER12}
{\sc L.~Diening, D.~Lengeler, B.~Stroffolini, A.~Verde}, {{Partial
  regularity for minimizers of quasi-convex functionals with general
  growth}}, {\it SIAM J. Math. Anal.} {\bf44} (2012), 3594--3616.

\bibitem{DUZGAS02}
{\sc F.~Duzaar and A.~Gastel}, {{Nonlinear elliptic systems with Dini
  continuous coefficients}}, {\it Arch. Math.} {\bf78} (2002), 58--73.

\bibitem{DUZGASGRO00}
{\sc F.~Duzaar, A.~Gastel, and J.~Grotowski}, {{Partial regularity for
  almost minimizers of quasi-convex integrals}}, {\it SIAM J. Math. Anal.}
  {\bf32} (2000), 665--687.

\bibitem{DUZGASMIN04}
{\sc F.~Duzaar, A.~Gastel, and G.~Mingione}, {{Elliptic systems, singular
  sets and Dini continuity}}, {\it Commun. Partial Differ. Equations} {\bf29}
  (2004), 1215--1240.

\bibitem{DUZGRO00}
{\sc F.~Duzaar and J.~Grotowski}, {{Optimal interior partial regularity for
  nonlinear elliptic systems: the method of A-harmonic approximation}}, {\it
  Manuscr. Math.} {\bf103} (2000), 267--298.

\bibitem{DUZGROKRO04}
{\sc F.~Duzaar, J.~Grotowski, and M.~Kronz}, {{Partial and full boundary
  regularity for minimizers of functionals with nonquadratic growth}}, {\it J.
  Convex Anal.} {\bf11} (2004), 437--476.

\bibitem{DUZGROKRO05}
{\sc F.~Duzaar, J.~Grotowski, and M.~Kronz}, {{Regularity of
  almost minimizers of quasi-convex variational integrals with subquadratic
  growth}}, {\it Ann. Mat. Pura Appl., IV. Ser.} {\bf11} (2005), 421--448.

\bibitem{DUZMIN08}
{\sc F.~Duzaar and G.~Mingione}, {{Harmonic type approximation lemmas}},
  {\it J. Math. Anal. Appl.} {\bf352} (2009), 301--335.

\bibitem{DUZMIN10A}
{\sc F.~Duzaar and G.~Mingione}, {{Gradient continuity
  estimates}}, {\it Calc. Var. Partial Differ. Equ.} {\bf39} (2010),
  379--418.

\bibitem{DUZMIN10B}
{\sc F.~Duzaar and G.~Mingione}, {{Gradient estimates
  via linear and nonlinear potentials}}, {\it J. Funct. Anal.} {\bf259}
  (2010), 2961--2998.

\bibitem{DUZMIN11}
{\sc F.~Duzaar and G.~Mingione}, {{Gradient estimates
  via non-linear potentials}}, {\it Am. J. Math.} {\bf133} (2011),
  1093--1149.

\bibitem{DUZSTE02}
{\sc F.~Duzaar and K.~Steffen}, {{Optimal interior and boundary regularity
  for almost minimizers to elliptic variational integrals}}, {\it J. Reine
  Angew. Math.} {\bf546} (2002), 73--138.

\bibitem{EVANS86}
{\sc L.~Evans}, {{Quasiconvexity and partial regularity in the calculus of
  variations}}, {\it Arch. Ration. Mech. Anal.} {\bf95} (1986), 227--252.

\bibitem{EVAGAR87}
{\sc L.~Evans and R.~Gariepy}, {{Blowup, compactness and partial regularity
  in the calculus of variations}}, {\it Indiana Univ. Math. J.} {\bf36}
  (1987), 361--371.

\bibitem{FUSHUT85}
{\sc N.~Fusco and J.~Hutchinson}, {{$\mathrm{C}^{1,\alpha}$ partial
  regularity of functions minimising quasiconvex integrals}}, {\it Manuscr.
  Math.} {\bf54} (1985), 121--143.

\bibitem{FUSHUT91}
{\sc N.~Fusco and J.~Hutchinson}, {{Partial regularity
  in problems motivated by nonlinear elasticity}}, {\it SIAM J. Math. Anal.}
  {\bf22} (1991), 1516--1551.

\bibitem{GARIEPY89}
{\sc R.~Gariepy}, {{A Caccioppoli inequality and partial regularity in the
  calculus of variations}}, {\it Proc. R. Soc. Edinb., Sect. A, Math.}
  {\bf112} (1989), 249--255.

\bibitem{giaquinta1983differentiability}
{\sc M.~Giaquinta and E.~Giusti}, {{Differentiability of minima of
  non-differentiable functionals}}, {\it Invent. Math.} {\bf72} (1983),
  285--298.

\bibitem{GIAGIU84}
{\sc M.~Giaquinta and E.~Giusti}, {{Sharp estimates for
  the derivatives of local minima of variational integrals}}, {\it Boll. Unione
  Mat. Ital., VI. Ser., A} {\bf3} (1984), 239--248.

\bibitem{GIAMOD79}
{\sc M.~Giaquinta and G.~Modica}, {{Almost-everywhere regularity results
  for solutions of non linear elliptic systems}}, {\it Manuscr. Math.} {\bf28}
  (1979), 109--158.

\bibitem{GIAMOD86}
{\sc M.~Giaquinta and G.~Modica}, {{Partial regularity
  of minimizers of quasiconvex integrals}}, {\it Ann. Inst. H. Poincar{\'e}
  Anal. Non Lin{\'e}aire} {\bf3} (1986), 185--208.

\bibitem{giusti1984}
{\sc E.~Giusti}, {{\it Minimal Surfaces and Functions of Bounded Variation}},
  Birkhäuser, 1984.

\bibitem{giusti2003direct}
{\sc E.~Giusti}, {{\it Direct Methods in
  the Calculus of Variations}}, World Scientific, 2003.

\bibitem{GIUMIR68}
{\sc E.~Giusti and M.~Miranda}, {{Sulla regolarit{\`a} delle soluzioni
  deboli di una classe di sistemi ellitici quasi-lineari}}, {\it Arch. Ration.
  Mech. Anal.} {\bf31} (1968), 173--184.

\bibitem{GIUMIR68EX}
{\sc E.~Giusti and M.~Miranda}, {{Un esempio di
  soluzioni discontinue per un problema di minimo relativo ad un integrale
  regolare del calcolo delle variazioni}}, {\it Boll. Unione Mat. Ital., IV.
  Ser.} {\bf1} (1968), 219--226.

\bibitem{GMEINEDER21}
{\sc F.~Gmeineder}, {{Partial regularity for symmetric quasiconvex functionals
  on $\mathrm{BD}$}}, {\it J. Math. Pures Appl. (9)} {\bf145} (2021), 83--129.

\bibitem{GMEKRI19}
{\sc F.~Gmeineder and J.~Kristensen}, {{Partial regularity for $\mathrm{BV}$
  minimizers}}, {\it Arch. Ration. Mech. Anal.} {\bf232} (2019), 1429--1473.

\bibitem{GMEKRI23}
{\sc F.~Gmeineder and J.~Kristensen}, {{Quasiconvex functionals of
  $(p,q)$-growth and the partial regularity of relaxed minimizers}}, {\it Arch.
  Ration. Mech. Anal.} {\bf248} (2024), 125 pages.

\bibitem{massari1994variational}
{\sc E.~Gonzalez and U.~Massari}, {{Variational mean curvatures}}, {\it
  Rend. Sem. Mat. Univ. Pol. Torino} {\bf52} (1994), 1--28.

\bibitem{GMT1993boundaries}
{\sc E.~Gonzalez, U.~Massari, and I.~Tamanini}, {{Boundaries of prescribed
  mean curvature}}, {\it Atti Accad. Naz. Lincei, Cl. Sci. Fis. Mat. Nat., IX.
  Ser., Rend. Lincei, Mat. Appl.} {\bf4} (1993), 197--206.

\bibitem{GUIDORZI02}
{\sc M.~Guidorzi}, {{Partial regularity in non-linear elasticity}}, {\it
  Manuscr. Math.} {\bf107} (2002), 25--41.

\bibitem{HAMBURGER96}
{\sc C.~Hamburger}, {{Partial regularity for minimizers of variational
  integrals with discontinuous integrands}}, {\it Ann. Inst. Henri
  Poincar{\'e}, Anal. Non Lin{\'e}aire} {\bf13} (1996), 255--282.

\bibitem{HAMBURGER98}
{\sc C.~Hamburger}, {{A new partial
  regularity proof for solutions of nonlinear elliptic systems}}, {\it Manuscr.
  Math.} {\bf95} (1998), 11--31.

\bibitem{HAMBURGER03}
{\sc C.~Hamburger}, {{Partial regularity
  of minimizers of polyconvex variational integrals}}, {\it Calc. Var. Partial
  Differ. Equ.} {\bf18} (2003), 221--241.

\bibitem{HAMBURGER07b}
{\sc C.~Hamburger}, {{Optimal partial
  regularity of minimizers of quasiconvex variational integrals}}, {\it ESAIM,
  Control Optim. Calc. Var.} {\bf13} (2007), 639--656.

\bibitem{HAMLUTMEAWHI07}
{\sc F.~Hammock, P.~Luthy, A.~Meadows, and P.~Whitman}, {{Tornado solutions
  for semilinear elliptic equations in $\mathbb{R}^2$: applications}}, {\it
  Proc. Am. Math. Soc.} {\bf135} (2007), 1419--1430.

\bibitem{HOPPER16}
{\sc C.~Hopper}, {{Partial regularity for holonomic minimisers of
  quasiconvex functionals}}, {\it Arch. Ration. Mech. Anal.} {\bf222} (2016),
  91--141.

\bibitem{KRIMIN05a}
{\sc J.~Kristensen and G.~Mingione}, {{The singular set of
  $\omega$-minima}}, {\it Arch. Ration. Mech. Anal.} {\bf177} (2005),
  93--114.

\bibitem{KRIMIN05}
{\sc J.~Kristensen and G.~Mingione}, {{The singular set of
  minima of integral functionals}}, {\it Arch. Ration. Mech. Anal.} {\bf180}
  (2006), 331--398.

\bibitem{KRIMIN07}
{\sc J.~Kristensen and G.~Mingione}, {{The singular set of
  Lipschitzian minima of multiple integrals}}, {\it Arch. Ration. Mech.
  Anal.} {\bf184} (2007), 341--369.

\bibitem{KUSMIN16}
{\sc T.~Kuusi and G.~Mingione}, {{Partial regularity and potentials}}, {\it
  J. {\'E}c. Polytech., Math.} {\bf3} (2016), 309--363.

\bibitem{maggi2012}
{\sc F.~Maggi}, {{\it Sets of Finite Perimeter and Geometric Variational
  Problems. An Introduction to Geometric Measure Theory}}, Cambridge University
  Press, 2012.

\bibitem{massari1974esistenza}
{\sc U.~Massari}, {{Esistenza e regolarit{\`a} delle ipersuperfici di
  curvatura media assegnata in $\mathbb{R}^n$}}, {\it Arch. Ration. Mech.
  Anal.} {\bf55} (1974), 357--382.

\bibitem{massari1975frontiere}
{\sc U.~Massari}, {Frontiere orientate
  di curvatura media assegnata in $ \mathrm{L}^p$}, {\it Rend. Sem. Mat. Univ.
  Padova} {\bf53} (1975), 37--52.

\bibitem{MINGIONE06}
{\sc G.~Mingione}, {{Regularity of minima: an invitation to the Dark Side
  of the Calculus of Variations}}, {\it Appl. Math.} {\bf51} (2006),
  355--425.

\bibitem{MOONEY20}
{\sc C.~Mooney}, {{Minimizers of convex functionals with small degeneracy
  set}}, {\it Calc. Var. Partial Differ. Equ.} {\bf59} (2020), p.~19 pages.

\bibitem{MONSA16}
{\sc C.~Mooney and O.~Savin}, {{Some singular minimizers in low dimensions
  in the calculus of variations}}, {\it Arch. Ration. Mech. Anal.} {\bf221}
  (2016), 1--22.

\bibitem{morrey1952quasi}
{\sc C.~Morrey}, {{Quasi-convexity and the lower semicontinuity of multiple
  integrals}}, {\it Pac. J. Math.} {\bf2} (1952), 25--53.

\bibitem{MORREY68}
{\sc C.~Morrey}, {{Partial regularity
  results for non-linear elliptic systems}}, {\it J. Math. Mech.} {\bf17}
  (1968), 649--670.

\bibitem{MUESVE03}
{\sc S.~M{\"u}ller and V.~{\v{S}}ver{\'a}k}, {{Convex integration for
  Lipschitz mappings and counterexamples to regularity}}, {\it Ann. Math. (2)}
  {\bf157} (2003), 715--742.

\bibitem{NECAS77}
{\sc J.~Ne{\v c}as}, {{Example of an irregular solution to a nonlinear
  elliptic system with analytic coefficients and conditions for regularity}}.
\newblock {Theor. Nonlin. Oper., Constr. Aspects, Proc. int. Summer Sch.,
  Berlin 1975, 197--206}, 1977.

\bibitem{PASSIE97}
{\sc A.~{Passarelli di Napoli} and F.~Siepe}, {{A regularity result for a
  class of anisotropic systems}}, {\it Rend. Ist. Mat. Univ. Trieste} {\bf28}
  (1996), 13--31.

\bibitem{PHILLIPS83A}
{\sc D.~Phillips}, {{A minimization problem and the regularity of solutions
  in the presence of a free boundary}}, {\it Indiana Univ. Math. J.} {\bf32}
  (1983), 1--17.

\bibitem{PHILLIPS83B}
{\sc D.~Phillips}, {{Hausdorff measure
  estimates of a free boundary for a minimum problem}}, {\it Commun. Partial
  Differ. Equations} {\bf8} (1983), 1409--1454.

\bibitem{SCHMIDT08}
{\sc T.~Schmidt}, {{Regularity of minimizers of
  $\mathrm{W}^{1,p}$-quasiconvex variational integrals with $(p,q)$-growth}},
  {\it Calc. Var. Partial Differ. Equ.} {\bf32} (2008), 1--24.

\bibitem{SCHMIDT09B}
{\sc T.~Schmidt}, {{A simple partial
  regularity proof for minimizers of variational integrals}}, {\it NoDEA,
  Nonlinear Differential Equations Appl.} {\bf16} (2009), 109--129.

\bibitem{SCHMIDT09A}
{\sc T.~Schmidt}, {{Regularity of
  relaxed minimizers of quasiconvex variational integrals with
  $(p,q)$-growth}}, {\it Arch. Ration. Mech. Anal.} {\bf193} (2009),
  311--337.

\bibitem{SCHMIDT14}
{\sc T.~Schmidt}, {{Partial regularity
  for degenerate variational problems and image restoration models in
  $\mathrm{BV}$}}, {\it Indiana Univ. Math. J.} {\bf63} (2014), 213--279.

\bibitem{SchmidtSchuettOptHoeldExp}
{\sc T.~Schmidt and J.~Sch{\"u}tt}, {{The optimal H{\"o}lder exponent in
  Massari’s regularity theorem}}, {\it Calc. Var. Partial Differ. Equ.}
  {\bf62} (2023), p.~23 pages.

\bibitem{SVEYAN00}
{\sc V.~{\v S}ver{\'a}k and X.~Yan}, {{A singular minimizer of a smooth
  strongly convex functional in three dimensions}}, {\it Calc. Var. Partial
  Differ. Equ.} {\bf10} (2000), 213--221.

\bibitem{SVEYAN02}
{\sc V.~{\v S}ver{\'a}k and X.~Yan}, {{Non-Lipschitz
  minimizers of smooth uniformly convex functionals}}, {\it Proc. Natl. Acad.
  Sci. USA} {\bf99} (2002), 15269--15276.

\bibitem{SZEKELYHIDI04}
{\sc L.~Sz{\'e}kelyhidi}, {{Non-Lipschitz minimizers of smooth uniformly
  convex functionals}}, {\it Arch. Ration. Mech. Anal.} {\bf172} (2004),
  133--152.

\bibitem{tamaninni1984regularity}
{\sc I.~Tamanini}, {Regularity results for almost minimal oriented
  hypersurfaces in $\mathbb{R}^n$}, {\it Quaderni del Dipartimento di
  Matematica dell’Universit\`a di Lecce} (1984), 95 pages.

\end{thebibliography}
\end{document}